\documentclass[11pt]{article}
\usepackage[usenames,dvipsnames]{xcolor}
\definecolor{OliveGreen}{rgb}{0,0.6,0}
\definecolor{tempblue}{RGB}{36, 56, 231 }
\usepackage{cite}
\usepackage{subfigure}
\usepackage{graphicx, color, graphpap}
\usepackage{amsmath,amssymb,amsthm,bm}
\usepackage{mathabx}
\usepackage{ifthen}
\usepackage{adjustbox}

\usepackage{matlab-prettifier}

\usepackage{mathrsfs}
\usepackage{mathtools}
\usepackage[shortlabels]{enumitem}

\usepackage{physics}
\usepackage{pstricks}
\usepackage{lineno}
\usepackage{float}
\usepackage{fancyhdr}
\usepackage[us,12hr]{datetime} 
\usepackage{exscale}
\usepackage{tabularx}
\usepackage{longtable}
\usepackage{latexsym}
\usepackage{fullpage}       
\usepackage{float}
\usepackage{verbatim}
\usepackage{multirow}
\usepackage{overpic}
\usepackage{comment}

\usepackage{siunitx}
\usepackage{colortbl}

\definecolor{tablegray}{RGB}{215, 219, 221 }

\usepackage{booktabs}
\usepackage{makeidx}

\usepackage{algorithm}   
\usepackage{algpseudocode}

\makeindex

\usepackage{csvsimple}

\usepackage[pagebackref=true]{hyperref}
\hypersetup{
	plainpages=false,       
	unicode=false,          
	pdftoolbar=true,        
	pdfmenubar=true,        
	pdffitwindow=false,     
	pdfstartview={FitH},    
	pdftitle={
Optimal Preconditioners; Iterative Methods
	           },  
	pdfnewwindow=true,      
	colorlinks=true,        
	linkcolor=blue,         
	citecolor=green,        
	filecolor=magenta,      
	urlcolor=cyan           
}

\numberwithin{equation}{section}  
\numberwithin{table}{section}
\numberwithin{figure}{section}
\numberwithin{algorithm}{section}

\usepackage{thmtools}
\usepackage[noabbrev]{cleveref}

\usepackage{mathtools}

\usepackage{refcount} 

\def\R{\mathbb{R}}
\def\Rp{\R_+}

\def\Sc{\mathbb{S}}
\def\Sn{\Sc^n}

\def\Sr{\Sc^r}
\def\Snp{\Sc_+^n}
\def\PSnp{P_{\Sc_+^n}}   
\def\Snn{\Sc_-^n}

\def\Srp{\Sc_+^r}

\def\Snpp{\Sc_{++}^n}

\def\Rn{\mathbb{R}^n}
\def\Rm{\mathbb{R}^m}

\def\Rnp{\mathbb{R}_+^n}

\def\cref#1{{\normalfont(\ref{#1})}}

\def\cref#1{{\normalfont(\ref{#1})}}

\newtheorem{theorem}{Theorem}[section]

\newtheorem{definition}[theorem]{Definition}
\newtheorem{example}[theorem]{Example}

\newtheorem{prop}[theorem]{Proposition}

\newtheorem{corollary}[theorem]{Corollary}

\newtheorem{remark}[theorem]{Remark}

\newtheorem{lemma}[theorem]{Lemma}

\crefname{thm}{Theorem}{Theorems}
\Crefname{thm}{Theorem}{Theorems}
\crefname{problem}{Problem}{Theorems}
\Crefname{problem}{Problem}{Theorems}
\Crefname{assump}{Assumption}{Theorems}
\crefname{assump}{Assumption}{Theorems}
\crefname{question}{Question}{Theorems}
\Crefname{question}{Question}{Theorems}
\crefname{conjecture}{Conjecture}{Theorems}
\Crefname{conjecture}{Conjecture}{Theorems}
\crefname{prop}{Proposition}{Propositions}
\Crefname{prop}{Proposition}{Propositions}
\crefname{cor}{Corollary}{Corollaries}
\Crefname{cor}{Corollary}{Corollaries}
\crefname{lem}{Lemma}{Lemmas}
\Crefname{lem}{Lemma}{Lemmas}
\theoremstyle{definition}
\crefname{defn}{definition}{definitions}
\Crefname{defn}{Definition}{Definitions}
\crefname{conj}{Conjecture}{Conjectures}
\Crefname{conj}{Conjecture}{Conjectures}
\crefname{remark}{Remark}{Remarks}
\Crefname{remark}{Remark}{Remarks}
\crefname{rmk}{Remark}{Remarks}
\Crefname{rmk}{Remark}{Remarks}
\crefname{example}{Example}{Examples}
\Crefname{example}{Example}{Examples}
\crefname{align}{}{}
\Crefname{align}{}{}
\crefname{equation}{}{}
\Crefname{equation}{}{}

\newcommand{\textdef}[1]{\textit{#1}\index{#1}}

\newcommand{\cX}{{\mathcal X} }

\newcommand{\cA}{{\mathcal A} }

\newcommand{\cR}{{\mathcal R} }

\newcommand{\cF}{{\mathcal F} }

\newcommand{\cV}{{\mathcal V} }

\newcommand{\beps}{\bm{\varepsilon}}

\newcommand{\cL}{{\mathcal L} }
\newcommand{\cS}{{\mathcal S} }
\newcommand{\cO}{{\mathcal O} }
\newcommand{\cOn}{\cO^n}

\newcommand{\MC}{\textbf{MC}\,}

\newcommand{\QAP}{\textbf{QAP}\,}
\newcommand{\QAPp}{\textbf{QAP}}

\newcommand{\DNN}{\textbf{DNN}\,}
\newcommand{\DNNp}{\textbf{DNN}}
\newcommand{\KKT}{\textbf{KKT}\,}
\newcommand{\KKTp}{\textbf{KKT}}

\def\BAP{\mbox{\bf BAP}\,}
\def\BAPp{\mbox{\bf BAP}}

\def\SDP{\mbox{SDP}\,}
\def\SDPp{\mbox{SDP}}

\def\LPp{\mbox{LP}}

\newcommand{\matlab}{\textsc{Matlab}\,}

\newcommand{\FR}{\textbf{FR}\,}
\newcommand{\FRp}{\textbf{FR}}

\newcommand{\CQ}{\textbf{CQ}\,}
\newcommand{\CQp}{\textbf{CQ}}

\def\Sk{\Sc^k}

\newcommand{\A}{{\mathcal A}}

\newcommand{\bbm}{\begin{bmatrix}}
\newcommand{\ebm}{\end{bmatrix}}
\newcommand{\bem}{\begin{pmatrix}}
\newcommand{\eem}{\end{pmatrix}}
\newcommand{\beq}{\begin{equation}}
\newcommand{\beqs}{\begin{equation*}}
\newcommand{\bet}{\begin{table}}
\newcommand{\eeq}{\end{equation}}
\newcommand{\eeqs}{\end{equation*}}
\newcommand{\beqr}{\begin{eqnarray}}

\renewcommand{\vec}{{\rm vec}}

\DeclareMathOperator{\supp}{{supp}}
\DeclareMathOperator{\cond}{{cond}}

\DeclareMathOperator{\lin}{lin}%
\DeclareMathOperator{\face}{face}
\DeclareMathOperator{\sd}{sd}
\DeclareMathOperator{\maxsd}{maxsd}
\DeclareMathOperator{\iips}{{ips}}

\DeclareMathOperator{\Nullity}{nullity}
\DeclareMathOperator{\nul}{null}
\DeclareMathOperator{\range}{range}

\DeclareMathOperator{\kvec}{{vec}}

\DeclareMathOperator{\blkdiag}{{blkdiag}}
\DeclareMathOperator{\diag}{{diag}}
\DeclareMathOperator{\Diag}{{Diag}}

\DeclareMathOperator{\svec}{{svec}}

\DeclareMathOperator{\relint}{{relint}}

\DeclareMathOperator{\spanl}{{span}}

\DeclareMathOperator{\odiag}{{o^0\diag}}
\DeclareMathOperator{\bdiag}{{b^0\diag}}

\newcommand{\nc}{\newcommand}
\nc{\arrow}{{\rm arrow\,}}
\nc{\Arrow}{{\rm Arrow\,}}
\nc{\BoDiag}{{\rm B^0Diag\,}}
\nc{\bodiag}{{\rm b^0diag\,}}

\nc{\Mm}{{\mathcal M}^{m} }
\nc{\Mmn}{{\mathcal M}^{mn} }
\nc{\Mnr}{{\mathcal M}_{nr} }
\nc{\Mnmr}{{\mathcal M}_{(n-1)r} }
\nc{\kwqqp}{Q{$^2$}P\,}
\nc{\kwqqps}{Q{$^2$}Ps}

\nc{\notinaho}{(X,S)\in \overline{AHO}(\A)}
\nc{\inaho}{(X,S)\in AHO(\A)}

\newcommand{\bea}{\begin{eqnarray}}%
\newcommand{\eea}{\end{eqnarray}}%
\newcommand{\beas}{\begin{eqnarray*}}%
\newcommand{\eeas}{\end{eqnarray*}}%
{}

\newcommand{\Hnp}[1][]{\,\mathbb{H}_+^{\ifthenelse{\equal{#1}{}}{n}{#1}}}
\newcommand{\Hn}[1][]{\,\mathbb{H}^{\ifthenelse{\equal{#1}{}}{n}{#1}}}
\newcommand{\Hk}[1][]{\,\mathbb{H}^{\ifthenelse{\equal{#1}{}}{k}{#1}}}
\newcommand{\Dn}[1][]{\,\mathbb{D}^{\ifthenelse{\equal{#1}{}}{n}{#1}}}

\newcommand{\<}{\langle}
\def\>{\rangle}

\DeclareMathOperator{\argmin}{\arg\min}

\usepackage{authblk}
\author{
Haesol Im}
\author{\href{https://uwaterloo.ca/combinatorics-and-optimization/about/people/group/50}{%
Woosuk L. Jung}}
\author{\href{https://cvnet.cpd.ua.es/curriculum-breve/es/torregrosa-belen-david/110216}{%
David Torregrosa-Bel\'en\footnote{
D.T.B was
partially supported by Centro de Modelamiento Matemático (CMM) BASAL fund FB210005 for center of excellence from ANID-Chile, Fondecyt Postdoctorado 3250039 and 
 by Grant PID2022-136399NB-C21 funded by ERDF/EU and by
MICIU/AEI/10.13039/501100011033.
}
}}
\author{\href{https://www.math.uwaterloo.ca/~hwolkowi}{%
Henry Wolkowicz}}
\affil{Department of Combinatorics and Optimization,
Faculty of Mathematics\\
University of Waterloo\\
Waterloo, Ontario, Canada N2L 3G1
}
\begin{document}
\linenumbers
\renewcommand{\thelinenumber}{} 
\title{{
Projection, Degeneracy, and Singularity Degree for Spectrahedra
}
}

\date{ \today \\
        Department of Combinatorics and Optimization\\
        Faculty of Mathematics, University of Waterloo, Canada. \\
Research partially supported by the Natural Sciences and Engineering Research
Council of Canada.
}
\maketitle

\tableofcontents
\listoffigures
\listoftables
\listofalgorithms

\begin{abstract}
Facial reduction, \FRp, is a regularization technique for convex
programs where the strict feasibility constraint qualification, \CQp, fails.
Though this \CQ holds generically, failure is pervasive in applications
such as semidefinite relaxations of hard discrete optimization problems.
In this paper we relate \FR
to the analysis of the convergence behaviour of a semismooth
Newton root finding method for the projection onto a spectrahedron, 
i.e.,~onto the intersection of a linear manifold and the semidefinite cone.
We examine the effect of failure of strict feasibility on the projection problem.
In the process, we derive  an elegant formula for the 
projection onto a face of the semidefinite cone obtained via regularization and discuss pathologies that arise in the absence of strict feasibility.
We show further that
the ill-conditioning of the Jacobian of the Newton method near optimality 
characterizes the degeneracy of the nearest point in the spectrahedron.
We apply the results, both theoretically and empirically,
 to the problem of finding nearest points to the
sets of: (i) correlation matrices or the \emph{elliptope}; and 
(ii) semidefinite relaxations of permutation matrices or the \emph{vontope},
i.e.,~the feasible sets for the semidefinite relaxations of the
max-cut and quadratic assignment
problems, respectively. 

\end{abstract}




\section{Introduction}
\label{sect:introd}

Facial reduction, \FRp, is a finite step process that regularizes
convex programs where the strict feasibility constraint qualification,
\CQp, fails. This \CQ holds generically for linear
conic programs, see e.g.,~\cite{MR3622250}.
However, failure is
pervasive in applications such as semidefinite programming, \SDPp,
relaxations of hard discrete optimization problems,
e.g.,~\cite{DrusWolk:16}. The minimum number of \FR steps
is denoted as the \emph{singularity degree of $\cF$, $\sd(\cF)$}, 
of the program with feasible set $\cF$. It has been
shown to be related to stability, error analysis, and convergence rates,
see e.g.,~\cite{S98lmi,SWW:17,MR4614122,DrusLiWolk:14}. 
Further generalized notions of singularity degree such as the maximum
number of \FR steps are studied
in~\cite{ImWolk:22,HaesolIm:2022} and shown to also relate to stability
and convergence rates.
In this paper we study $\sd(\cF)$ and relations to the
projection problem, or best approximation problem (\BAPp), 
onto a \textdef{spectrahedron}, the
intersection of a linear manifold and the positive semidefinite cone in
symmetric matrix space.

Our main purpose is to examine the effect of failure of strict feasibility 
on the projection problem. In the absence of strict feasibility,
we find surprising relationships between the eigenpairs of 
small eigenvalues of the Jacobian
in a newly proposed Newton method for the projection problem and finding 
exposing vectors for \FRp.
Furthermore, we provide a characterization that links the degeneracy of a point $\bar{X}$ in the spectrahedron, which pertains solely to the feasible set, to the singularity of the Jacobian computed in the course of the Newton method applied to the \BAPp, where $\bar{X}$ is the projected point. 
This establishes a connection between the degeneracy and the behaviour of the Newton method.
We apply the results, both theoretically and empirically,
to the problem of finding nearest points to the
sets of: (i) correlation matrices or the elliptope; and 
(ii) semidefinite relaxations of permutation matrices or the vontope,
i.e.,~the feasible sets for the semidefinite relaxations of the
max-cut and quadratic assignment
problems, respectively.

\index{$\sd(\cF)$, singularity degree of $\cF$}
\index{singularity degree of $\cF$, $\sd(\cF)$}

\subsection{Projection Problem}
\label{sect:BAPprob}
\index{$X^*$, optimum}
\index{optimum, $X^*$}
\index{$p^*$, optimal value}
\index{optimal value, $p^*$}

\index{\BAPp, best approximation problem}
\index{$W\in \Sn$, data}

We work with the Euclidean space of $n\times n$ real symmetric matrices,
$\Sn$, equipped with the trace inner product.  
Let the \textdef{data, $W\in \Sn$}, be given.
The projection, or basic \textdef{best approximation problem, \BAPp}, is
\begin{equation}
\label{eq:DNNparam}
\begin{array}{rcll}
X^* =   & \argmin & \frac 12 \|X-W\|^2,
               & \quad \qquad \textdef{$p^* = \frac 12 \|X^*-W\|^2$},
\\ & \text{s.t.} &  X\in \cF := \cL \cap \Snp,
\end{array}
\end{equation}
where $\Snp\subseteq \Sn$ is the closed convex cone of positive
semidefinite matrices in the vector space of real symmetric matrices of
order $n$, equipped with the trace inner product. Given a convex cone 
$\cX$, we let \textdef{$X\succeq_\cX 0$} denote $X\in \cX$; and we 
often use $X \succeq 0$ for $X\succeq_{\Snp} 0$ when the meaning is clear.
Here $\cL\subseteq \Sn$ is a linear manifold;
and, $p^*, X^*$ are the optimal value and optimum, respectively.
The representation of the linear manifold is essential in algorithms and
different representations can result in different stability properties
for the problem, e.g.,~\cite{ScTuWominimal:07}. 
We let $\cL = \{X \in \Sn : \cA X = b\}$, where 
$\cA : \Sn\to \Rm$ is a given surjective (without loss of generality) linear
transformation;
$\cA X = (\trace A_iX)\in \Rm$ for given fixed linearly
independent $A_i\in \Sn, i=1,\ldots,m$.
We assume that $\cF$ is a nonempty feasible set;
it is called a \textdef{spectrahedron}.
Here the data of \BAP is $W,\cA,b$.

\index{$\cF = \cL\cap K$, feasible set}
\index{feasible set, $\cF = \cL\cap K$}

Nearest point problems are pervasive in the literature and are
often the essential step in feasiblity seeking problems,
e.g.,~\cite{MR4675665,MR3374759,MR3155358}. We study these
problems and show that they reveal hidden structure and information about
the stability and conditioning of feasible sets and the degeneracy of
optimal points.
Related convergence analysis and new types of
singularity degree are given in~\cite{DrusLiWolk:14,ImWolk:22}.
Recall that a \textdef{correlation matrix} 
is a positive semidefinite matrix with
diagonal all one. The set of correlation matrices is often called the
\textdef{elliptope}.
Finding the nearest correlation matrix is one application
\cite{MR3537883,MR1918653,MR2580548} that arises in many areas,
e.g.,~finance.
In addition, we specifically look at the feasible set of the
\SDP relaxation of the quadratic assignment problems \QAPp, 
which we call the \textdef{vontope}.
We characterize degeneracy of nearest points
and the resulting effects on stability of the nearest point algorithm
for these two special instances.

\subsubsection{Related Results}

The \BAP for the polyhedral case is studied in
\cite{CensorMoursiWeamsWolk:22} with application to linear programming.
(In the linear programming, \LPp, case, $\Sn \leftarrow \Rn, \Snp
\leftarrow \Rnp$.)
Generalized Jacobians play a critical role, though the relation to
stability is not studied.
The \SDP case is studied in
e.g.,~\cite{HenrionMalick:11,MR2112861}.
They use a quasi-Newton method to solve a dual problem similar to our
dual problem; though we use a regularized
semismooth Newton method with a  generalized Jacobian and illustrate
fast quadratic convergence for well-posed problems.
Further related results on spectral functions, projections, and
Jacobians, appear in \cite{MR2252652}.

In \cite{im2024implicit} it is shown that 
\emph{any} conic program that fails strict
feasibility has implicit redundancies and every point is degenerate.
Relationships with the Barvinok-Pataki bound and strengthened
bound~\cite{GPat:95,MR1797294,ImWolk:21}
for conic programs is discussed.
Further discussions on degeneracy related to 
loss of strict complementarity appear in \cite{MR3301316}.

\index{method of alternating projections}
The paper \cite{DrusLiWolk:14} provides a sublinear upper bound based on
the singularity degree for the convergence rate of the method of
alternating projections applied to spectrahedra. The 
\linelabel{line:citesingleton} \label{pg:citesingleton}
paper~\cite{MR4832118}
(preprint was published as we were finishing the preparation of this manuscript) furnishes analytic formulas for the sequence generated by alternative projection that reveal  that this upper bound can fail to be tight. Although this work is not restricted to the singleton case, much of the analysis is developed with emphasis on the singleton setting. 
Further results on accuracy and differentiability appear 
in~\cite{GOULART2020177,MR2252652}.

\subsection{Outline}

We continue in \Cref{sect:background} with the background of projections and an overview of \FR focused on the connections to degeneracy and strict feasibility. 
\Cref{sec:char_optCond} presents the equation to be solved using the
semismooth Newton root finding method, which is derived through a detailed examination of the optimality conditions along with notions on facial structure. 
We also derive an useful formula for the projection
onto a \emph{face} of the semidefinite cone.
In \Cref{sec:basicNewton} we introduce the semismooth Newton method for the \BAP
and derive Jacobians associated with the Newton method.
In \Cref{sec:FailureReg}, we investigate various phenomena arising in the absence of strict feasibility. This includes
the relationships between the eigenpairs of small eigenvalues of the Jacobian in the Newton method and identification of exposing vectors for \FR (\Cref{sec:Patho});
the singularity of the Jacobian and its link to degeneracy (\Cref{sec:JacoDegen}); and applications to the set of correlation matrices and the  feasible set of the \SDP relaxations of the \QAP~(\Cref{sec:degenofMCQAP}).
We present numerical experiments in~\Cref{sec:Numerics}
 with results and observations  discussed in \Cref{sec:FailureReg}. Our concluding remarks are in~\Cref{sect:concl}.

\section{Background}
\label{sect:background}
We first present some background on projections and 
related spectral functions, and then include the notions of facial
reduction for regularization, singularity, and degeneracy.

\subsection{Derivatives of Projection Operators}

We follow the work and notation in
\cite{Lew:96,Fr:81,MR2252652,ParikhBoyd:13,Yu2014ThePO}. We work with $f : \Sn \to \R \cup \{+\infty\}$, a closed proper 
extended valued convex function on $\Sn$.
We denote \textdef{$P_S$, projection onto a nonempty closed convex 
set $S$}, i.e.,
\[
P_S(W) = \argmin_{X\in S} \frac 12\|W-X\|^2.
\]
And for a convex set $S$, we denote the 
\textdef{indicator function, $\iota_S$}.
  The \emph{Moreau regularization of $\iota_{\Snn}$}, $\Snn:=-\Snp$,  (see \cite{ParikhBoyd:13}) is given by 
\begin{equation}
\label{eq:Deltaofz}
\Delta(X) = \min_{Z\in \Sn} \left\{\frac 12 \|X-Z\|^2 + \iota_{\Snn}(Z)\right\}.
\end{equation}

\index{projection onto closed convex set $S$, $P_S$}
\index{$\iota_S$, indicator function}
\index{$\cOn$, orthogonal matrices}
\index{orthogonal matrices, $\cOn$}

A \textdef{spectral function} $g:\Sn\to  \R\cup \{+\infty\}$ is one that is
invariant under orthogonal conjugation (congruence)
\[
g(X) = g(U^TXU), \, \forall\,  X\in \Sn, \forall\, U \in \cOn,
\]
where $\cOn$ is the set of orthogonal matrices of order $n$.
In \cite[Lemmas 2.3-4]{MR2252652}, it is shown that the Moreau
regularization, $\Delta(X)$, of $\iota_{\Snn}$ is a
\textdef{spectral function} with its gradient
given as
\index{Moreau regularization of $\iota_{\Snn}$, $\Delta(X)$}
\index{$\Delta(X)$, Moreau regularization of $\iota_{\Snn}$}
$\nabla \Delta (X) = P_{\Snp}(X)$.
See~\Cref{lem:diffspectr} below.
Let $\lambda: \Sn \to \Rn$ denote the eigenvalue function,
i.e.,~$\lambda(X)$ is the vector of eigenvalues of $X\in \Sn$ in  nonincreasing order.

\begin{lemma}[{\cite[Lemma 2.3]{MR2252652}}]
The function $\Delta$ in \cref{eq:Deltaofz} is the spectral function $\Delta = \delta \circ
\lambda$, where the function $\delta$ on $x\in \Rn$ is:
\linelabel{line:lemform} \label{pg:lemform}
\[
\delta(x) = \frac 12\sum_{i=1}^n \max\{0,x_i\}^2.
\]
\end{lemma}
\begin{proof}
We include the proof from \cite[Lemma 2.3]{MR2252652} for completeness.
For any $X\in \Sn$ we have
\[
\begin{array}{rcl}
\Delta(X) 
&=&
 \frac 12 \|X- P_{\Snn}(X)\|^2
\\&=&
 \frac 12 \| P_{\Snp}(X)\|^2
\\&=&
 \frac 12 \sum_{i=1}^n \max\{0,\lambda_i(X)\}^2
\\&=&
\delta(\lambda(X)).
\end{array}
\]
\samepage
\end{proof}

\begin{lemma}[{\cite[Lemma 2.4]{MR2252652}}]
\label{lem:diffspectr}
The function $\Delta$ in~\cref{eq:Deltaofz} is convex and differentiable.
Moreover, its gradient at $X\in \Sn$ is $P_{\Snp}(X)$, i.e.,
\[
\Delta(X)^\prime(dX) = \langle \nabla \Delta(X),dX\rangle =
         \langle P_{\Snp}(X),dX\rangle  =
         \trace P_{\Snp}(X)dX.
\]
\end{lemma}
\linelabel{line:spectrfns} \label{pg:spectrfns}
From spectral function theory e.g.,~\cite{Lew:96,Fr:81,MR2252652}, 
we know that the differentiability in~\Cref{lem:diffspectr}
follows from the differentiability of $\delta
\circ \lambda$. The formula for the derivative follows from the spectral
function formula with $X=U\Diag(\lambda(X))U^T$:
\[
\nabla (\delta\circ \lambda)(X) = U\left(\Diag \nabla
\delta(\lambda(X))\right)U^T, \  U \in \cOn.
\]
The derivative (Jacobian) of the projection can be found from the
Hessian of the regularization function
\[
P_{\Snp}^\prime(X) = \nabla^2 \Delta (X).
\]

\subsection{Facial Reduction}

The facial structure of cones plays an essential role when analyzing
various stability concepts.
In this section we study various properties that arise from the absence
of strict feasibility.  
\Cref{sec:RegforSD} presents the basics of the facial structure of $\Snp$ and the facial reduction process for $\cF$. 
In \Cref{sec:SDdiscussion} we revisit known notions of singularities 
and the length of the facial reduction process that connects to the dimension of the solution set of our problem \eqref{eq:DNNparam}.

\subsubsection{Facial Reduction Process}
\label{sec:RegforSD}

\index{face, $K\unlhd C$}
\index{$K\unlhd C$, face}
\index{proper face}
\index{conjugate face of $K$, $K^{\Delta}$}
\index{$C^+$, nonnegative polar cone of $C$}
\index{$C^{\perp}$, orthogonal complement of $C$}

We first review the basic facial structure of $\Snp$.
Recall that the convex cone~$f\subseteq K$ is a face of a convex cone $K\subseteq \Sn$, denoted by $f\unlhd K$, if 
\[
x,y \in K, z= x + y, z \in f \,\, \implies  \,\, x,y \in f.
\]
The cone $f$ is a proper face if $\{0\}  \subsetneq  f  \subsetneq K$.
Here we denote \textdef{$f^{\Delta}$, conjugate face of $f$}, defined as
\linelabel{line:nonnegpolar} \label{pg:nonnegpolar}
the intersection of the nonnegative polar cone with the orthogonal complement,
$f^{\Delta} = f^{\perp} \cap K^+$, where 
$K^+ :=\left\{  \phi : \langle \phi, k \rangle \geq 0, \forall k \in K\right\}$
is the \emph{nonnegative polar cone} of $K$.
The facial structure of $\Snp$ is well-studied 
and has an intuitive characterization. 
For any convex set $C\subseteq \Snp$, the minimal face of $\Snp$ containing $C$,
i.e.,~the intersection of all faces of $\Snp$ containing $C$, is denoted
$\face(C,\Snp)$. For the singleton $C=\{X\}\subseteq \Snp$, we get
\[
\textdef{$\face(X,\Snp)$} = \{ Y \in \Snp : \range(Y) \subseteq \range(X)  \}.
\]
Facial reduction (\FRp) for $\cF$, the constraint system in
\eqref{eq:DNNparam},  is the process of identifying the minimal face of $\Snp$ containing $\cF$. 
It is known that a point $\hat{X}$ in the relative interior of $\cF$, $\relint (\cF)$, provides the following characterization 
\[
\face(\hat{X},\Snp) = \face(\cF, \Snp). 
\]
Furthermore,  $\face(\hat{X},\Snp) = V \Srp V^T$ for some rank-$r$ matrix $V$ satisfying $\range(V) = \range(\hat{X})$. Such a matrix $V$ is called a
\textdef{facial range vector} for $\face(\hat{X},\Snp)$.

\index{relative interior, $\relint$}
\index{$\relint$, relative interior }

Finding $\face(\cF, \Snp)$ for an arbitrary $\cF$ analytically is a
challenging task. An alternative approach that is often used is to find the
minimal face numerically.
\Cref{prop:thmalt} below provides a tool for the \FR process.
Below, $\cA^*$ refers to the \textdef{adjoint} of $\cA$.

\index{facial reduction, \FRp}
\index{\FRp, facial reduction}

\begin{prop}[theorem of the alternative]
\label{prop:thmalt}
For the feasible constraint system $\cF$ defined in \cref{eq:DNNparam}, 
exactly one of the following statements holds:
\begin{enumerate}[(i)]
\item there exists $X \succ 0$ such that $X\in \cL$;
\item there exists $\lambda \in \Rm$ that satisfies
the \textdef{auxiliary system}
\end{enumerate}
\begin{equation}
\label{eq:AuxSys}
0\neq Z = \cA^*\lambda\succeq 0, \ \langle b, \lambda \rangle = 0 .
\end{equation}
\end{prop}

The vector $Z$ in \cref{eq:AuxSys} is called an
\textdef{exposing vector} for the face $(Z^\perp \cap \Snp) \supseteq
\cF$, as $X\in \cF$ implies 
\[
0=\langle b, \lambda \rangle = \langle \cA X,\lambda \rangle =  \langle
X, \cA^*\lambda \rangle = \langle X,Z\rangle.
\]
Here the face $(Z^\perp \cap \Snp)$ is 
exposed by the hyperplane $Z^\perp = \{Z\}^\perp$, and this allows for
a simplified expression of feasible points.
This restriction results in an equivalent strictly smaller dimensional problem.
Finding a solution to \eqref{eq:AuxSys} does not necessarily lead
directly to the minimal face, i.e.,~we can have 
$(Z^\perp \cap \Snp) \supsetneq \face(\cF,\Snp)$.
For such a case, the process is reapplied until $\face(\cF,\Snp)$ is
found. The number of times the system \eqref{eq:AuxSys} needs to be
solved varies, though it is at most $\min\{m,n\}$,
see~\Cref{sec:SDdiscussion} below.
A reader may refer to \Cref{example:DiffSolSet} for a brief illustration of how the theorem of the alternative is used for the \FR process.

\subsubsection{Three Notions of Singularity Degree}
\label{sec:SDdiscussion}

We now exhibit some properties that originate from the
number of \FR iterations.

\index{$\sd(\cF)$, singularity degree of $\cF$}
\index{singularity degree of $\cF$, $\sd(\cF)$}
\index{$\maxsd(\cF)$, max-singularity degree of $\cF$}
\index{max-singularity degree of $\cF$, $\maxsd(\cF)$}
\begin{definition}
\label{def:maxsd}
The \emph{singularity degree} of $\cF$, denoted $\sd(\cF)$, is the
\underline{minimum} number of \FR
iterations for finding $\face(\cF, \Snp)$.
The \emph{maximum singularity degree} of $\cF$, denoted
$\maxsd(\cF)$, is the \underline{maximum} number of nontrivial \FR iterations for finding~$\face(\cF, \Snp)$.
\end{definition}
\index{\FRp, facial reduction} 

\index{$\iips$, implicit problem singularity} 
\index{implicit problem singularity, $\iips$} 

The singularity degree is often used to relate error bounds to explain
the difficulty of solving problems numerically; see~\cite{S98lmi,SWW:17}.
It is known that a high singularity degree results in a worse 
\emph{forward error bound} relative to the backward errors.
The maximum singularity degree is a relatively new notion, and this
motivates the  idea of \emph{implicit problem singularity, $\iips(\cF)$}.
Every nontrivial step of \FR results in redundant linear constraints.
More specifically, \FR reveals redundant equalities in $\cA(\cdot)=b$; 
see \cite{HaesolIm:2022}.
The total number of these implicitly redundant constraints is called
$\iips(\cF)$ and a short argument shows that $\iips(\cF) \geq \maxsd(\cF)$.
\Cref{prop:FRsequence} below 
uses $\maxsd(\cF)$ to extend the result in  \cite[Lemma 3.5.2]{Sremac:2019}
and shows an interesting property that a \FR
sequence generates.\footnote{We note that the concepts of $\maxsd(\cF), \iips(\cF)$
did not yet exist in \cite{Sremac:2019}. Moreover, it is shown
empirically in \cite{HaesolIm:2022} that $\iips$ is directly related to
the forward error for \LPp s.}

\linelabel{line:propFR} \label{pg:propFR}
\begin{prop}\cite[Theorem 1]{liu2018exact},\cite[Lemma 3.5.2]{Sremac:2019}
\label{prop:FRsequence}
Let the exposing vector obtained in the $i$-th \FR iteration be
nontrivial, $Z^i = \cA^*(\lambda^i)\neq 0, i=1,\ldots,k\leq\maxsd(\cF)$.
Then the vectors, $\lambda^1, \lambda^2, \ldots, \lambda^k$, 
are linearly independent.
\end{prop}
We relate \Cref{prop:FRsequence} to the dimension of the set of the roots that arise in the algorithm for solving the \BAPp, 
e.g.,~see~\Cref{thm:maxsdSolutions}.

\subsection{Degeneracy and Relations to Strict Feasibility}
\label{sec:degen}

Many of the concepts of degeneracy arise from the early work with the simplex
method for linear programs. The \emph{stalling} phenomenon of the simplex
method is a well-known subject of research 
\cite{MR0056264,MR1885204,MR850380} and
many methods are proposed to overcome this difficulty.
In this section we use a generalized definition of degeneracy proposed by 
Pataki \cite{PatakiSVW:99} to extend the discussion to
spectrahedra. We then examine a connection between the \textdef{Slater
constraint qualification} (strict feasibility), and degeneracy 
of feasible points. 
We identify a type of degeneracy that inevitably arises in the absence of strict feasibility.

\begin{definition}\cite[Definition 3.3.1]{PatakiSVW:99}
\label{def:NondegenDef}
A point ${X} \in \cF$ is called \textdef{nondegenerate} if
\[
\lin ( \face({X}, \Snp )^{\Delta}) \cap \range (\cA^*) = \{0\}.
\] 
\index{degenerate point}
\end{definition}
\Cref{def:NondegenDef} immediately yields \Cref{lemma:NondegenChar}.
\begin{lemma}\cite[Corollary 3.3.2]{PatakiSVW:99}
\label{lemma:NondegenChar}
Let ${X} \in \cF$ and let
\begin{equation}
\label{eq:spectrX}
{X} = \begin{bmatrix} V &\bar{V} \end{bmatrix}
\begin{bmatrix} D & 0 \\ 0 & 0 \end{bmatrix}
\begin{bmatrix} V &\bar{V} \end{bmatrix}^T, \, D\succ 0,
\end{equation}
be a spectral decomposition of ${X}$. Then 
\begin{equation}
\label{eq:RorateAis}
\begin{array}{c}
	{X} \text{  is a nondegenerate point of $\cF$}
\\ \text{if, and only if,}
\\\left\{
\begin{bmatrix}
V^T A_i V & V^T A_i \bar{V} \\
\bar{V}^T A_i V & 0
\end{bmatrix}
\right\}_{i=1}^m 
\text{  is a set of linearly independent matrices in  } \Sn.
\end{array}
\end{equation}
\end{lemma}

\index{triangular number, $t(n) = n(n+1)/2$}

\begin{remark}
\label{rem:characdeggeneric}
The degeneracy of a point $X\in \cF$  can be identified by constructing a matrix 
using the characterization~\eqref{eq:RorateAis}.
We let $W_i:=
\begin{bmatrix}
V^T A_i V & V^T A_i \bar{V} \\
\bar{V}^T A_i V & 0
\end{bmatrix},\, i=1,\ldots,m$ and 
denote \textdef{$t(n) = n(n+1)/2$, triangular number}.
We let \textdef{$\vec$} denote the vectorization of a matrix by column;
and \textdef{$\svec : \Sk \to \R^{t(k)}$}
 denotes the isometry that vectorizes a
symmetric matrix using the upper triangular part after multiplying the
strict upper-triangular part by $\sqrt 2$. 
We construct the following matrix $L \in \R^{t(n)\times m}$ with an
appropriate permutation matrix $\Pi $ and the $i$-th column:
\[
Le_i = \textdef{$\Pi \svec W_i$}=
\begin{pmatrix}\svec V^T A_i V 
      \cr \sqrt 2 \kvec V^T A_i \bar{V} \cr
\svec 0 
\end{pmatrix}, \ \forall i \in \{1,\ldots, m\},
\]
where $V$ and $\bar{V}$ are given in \Cref{lemma:NondegenChar}.
Consider the matrix $\bar{L}$, with the $i$-th column   
$\bar{L} e_i = 
\begin{pmatrix}\svec V^T A_i V 
      \cr \sqrt 2\kvec V^T A_i \bar{V} \cr
\svec \bar{V}^T A_i \bar{V}
\end{pmatrix}$. Note that $\bar{L}$ is full-column rank given $\cA$ is surjective.
The matrix $L$ is obtained after zeroing out the last $t(\Nullity(X))$ rows of $\bar{L}$.
We note that $\rank(L)< \rank(\bar{L})$ (i.e., degeneracy holds) 
if, and only if, the orthogonal complement of the
span of the first
 $t(n)-t(\Nullity(X))$ rows of $\bar{L}$ has nonzero intersection with the span of the remaining rows
that are then changed to $0$.
Therefore, if $t(n) > m+t(\Nullity(X))$, then generically nondegeneracy holds.
\end{remark}

\linelabel{lem:AiVdep} \label{pg:AiVdep}
\begin{theorem}
\label{thm:AiVdep}
Suppose that $\cF$ fails strict feasibility and let $X= VDV^T \in \cF$
be found using~\cref{eq:spectrX}.
Then the set $\{A_iV\}_{i=1}^m$ is linearly dependent.
In particular, any solution $\lambda$ to~\eqref{eq:AuxSys} certifies the linear dependence of the set $\{A_iV\}_{i=1}^m$.
Thus every point in $\cF$ is degenerate. 
\end{theorem}
\begin{proof}
Suppose that $\cF$ fails strict feasibility. 
Let $X= VDV^T \in \cF$ and let $\lambda$ be a solution to the 
\textdef{auxiliary system} \eqref{eq:AuxSys}.  
Then~$\cA^*(\lambda)$ is an exposing vector and hence 
the inner product $\langle \cA^*(\lambda),X\rangle = 0$ which yields
\begin{equation}
\label{eq:AiV_relation}
0=\cA^*(\lambda)X= \left(\cA^*(\lambda)V\right)DV^T \implies
   0=\cA^*(\lambda) V = \sum_{i=1}^m \lambda_i \left(A_i V\right) .
\end{equation}
Since $\lambda$ is a nonzero vector, \eqref{eq:AiV_relation} shows the
linear dependence.

The linear dependence of the set $\{A_iV\}_{i=1}^m$ shown above
allows for verifying \emph{total} degeneracy that occurs in the absence
of strict feasibility of $\cF$.
Then the above provides $\sum_{i=1}^m \lambda_i A_i V =0,
\lambda \neq 0$.
We observe that
\[
\begin{array}{rclcl}
0 
& = &  V^T \left( \sum_{i=1}^m \lambda_i A_i V \right)  
& = & \sum_{i=1}^m \lambda_i V^T A_i V , \\
0 
& = &  \bar{V}^T \left( \sum_{i=1}^m \lambda_i A_i V \right)  
& = & \sum_{i=1}^m \lambda_i \bar{V}^T A_i V  . \\
\end{array}
\]
This immediately implies that the matrices in \eqref{eq:RorateAis} are linearly dependent and hence $X$ is degenerate. 
\samepage
\end{proof}

In \Cref{coro:degen_coros} we now connect nondegeneracy to strict feasibility. 
\begin{corollary}
\label{coro:degen_coros}
Let $\cF$ be given. Then the following holds:
\begin{enumerate}[(i)]
\item \label{item:FSlaterwhenNondeg}
If $\cF$ contains a nondegenerate point, then strict feasibility holds.
\item \label{item:posXnondegen}
Every $X\in \cF\cap \Snpp$ is nondegenerate.
\end{enumerate}
\end{corollary}

\begin{proof}
\Cref{item:FSlaterwhenNondeg} is the contrapositive of \Cref{thm:AiVdep}.
\Cref{item:posXnondegen} is immediate from the definition of
nondegeneracy,~\Cref{def:NondegenDef}, since 
$\face(X,\Snp)^\Delta = 0$, for all $X\succ 0$.
\samepage
\end{proof}

\Cref{prop:twonondeg,prop:relintNodegen} below provide sufficient
conditions for the nondegeneracy of certain points of spectrahedra. They
enable identifying nearest points where the semismooth Newton method for the \BAP performs effectively.

\begin{prop}
\label{prop:twonondeg}
Let $X_1, X_2\in \cF$  and let $X_1$ be a nondegenerate point. 
Then, $\gamma X_1 + (1-\gamma) X_2$ is a nondegenerate point of $\cF$,
for all $\gamma \in (0,1]$.
\end{prop}
\begin{proof}
Let $X_1, X_2\in \cF$  and let $X_1$ be a nondegenerate point. 
Let $\gamma \in (0,1]$ and  $X^\prime = \gamma X_1 + (1-\gamma) X_2$.
We observe that
\[
\begin{array}{crcl}
&\face(X_1,\Snp) &\subseteq& \face(X^\prime, \Snp) \\
\implies &  
\face(X_1,\Snp)^{\Delta}  &\supseteq& \face(X^\prime, \Snp)^{\Delta} \\
\implies &  
\lin \left( \face(X_1,\Snp)^{\Delta}  \right) &\supseteq & \lin \left( \face(X^\prime, \Snp)^{\Delta} \right) \\
\implies &  
\lin \left( \face(X_1,\Snp)^{\Delta}  \right) \cap \range(\cA^*) &\supseteq & \lin \left( \face(X^\prime, \Snp)^{\Delta} \right) \cap \range(\cA^*) .\\
\end{array}
\]
Since $X_1$ is a nondegenerate point, we have $\lin \left( \face(X_1,\Snp)^{\Delta}  \right) \cap \range(\cA^*) = \{0\}$.
Thus,~$X^\prime$ is a nondegenerate point.
\samepage
\end{proof}

\begin{prop} 
\label{prop:relintNodegen}
Let $f$ be a face of $\cF$ containing a nondegenerate point. Then every point in $\relint(f)$ is nondegenerate.  
\end{prop}

\begin{proof}
Let $X_1\in f$ be a nondegenerate point. For any $X \in\relint(f)$ there exists $X_2$ such that $X$ belongs to the segment $(X_1,X_2)$. The nondegeneracy of $X$ then follows from  \Cref{prop:twonondeg}.
\samepage
\end{proof}

\section{Characterization of Optimality for the \BAP}
\label{sec:char_optCond}

We now study the optimality conditions of the \BAP \Cref{eq:DNNparam} and present several properties,
including an equation for the application of Newton's method.

\begin{theorem}
\label{th:duality}
Consider the projection problem \cref{eq:DNNparam}. Then the following hold:
\begin{enumerate}[(i)]
 \item\label{it:thduality-1} $p^*$ is finite and the optimum $X^*$ exists and is unique.
 \item\label{it:thduality-2} There is a zero duality gap between the
primal and the dual problem of~\cref{eq:DNNparam}, where the Lagrangian dual is
the maximization of the dual functional, $\phi(y,Z)$, i.e.,
  \index{dual functional, $\phi(y,Z)$}
  \index{$\phi(y,Z)$, dual functional}
\begin{equation}
\label{eq:zerogap}
\begin{array}{rcl}
  p^* = d^* &=&   \max_{Z\succeq 0, y\in\Rm}  -\frac 1 2
\|Z+\cA^* y\|^2+\langle y, b-\cA W \rangle -  \langle Z, W\rangle.
\end{array}
\end{equation}
 \item\label{it:thduality-3} 
Strong duality (zero duality gap and dual attainment) 
holds in~\Cref{eq:DNNparam} if, and only if, there exists a root
$\bar y, F(\bar y)=0$, of the function 
\index{$F( y): = \cA P_{\Snp}(W+\cA^*  y) - b$}
\begin{equation}
\label{eq:def_of_F}
F( y): = \cA P_{\Snp}(W+\cA^*  y) - b.
\end{equation}
Moreover, in this case the solution to the primal 
problem is given by
\linelabel{line:Zpsdnot} \label{pg:Zpsdnot}
\[
X^* = P_{\Snp}(W+\cA^* \bar y).
\]
\end{enumerate}
\end{theorem}
\begin{proof}
\Cref{it:thduality-1}: The primal problem~\cref{eq:DNNparam} is
the minimization of a strongly convex function over a nonempty closed
convex set. This yields that the optimal value is finite and  is
attained at a unique point.

\Cref{it:thduality-2}: Since the primal objective function is
coercive, there is a zero duality gap, see 
e.g.,~\cite[Theorem 5.4.1]{MR1931309}.
Let $Z\succeq 0$. The Lagrangian function of problem~\cref{eq:DNNparam}
and its gradient are given by
\[
L(X,y,Z) = \frac 12\|X-W\|^2 +\langle y,b-\cA X\rangle -\langle Z,X\rangle, 
\quad
\nabla_X L(X,y,Z) = X-W -\cA^*y  - Z.
\]
It follows that $X$ is a stationary point of the Lagrangian if 
\begin{equation}
\label{eq:elimXfromLagr}
X = W + \cA^*y+Z.
\end{equation}
\label{pg:inff=g}
\linelabel{line:inff=g}
By means of this equality, we can eliminate $X$ and
 write the Lagrangian dual as the following maximization problem in
$Z,y$.
	\begin{equation*}
\begin{aligned}
d^* &= \max_{Z\succeq 0,y} \min_{X} L(X,y,Z) \\
& =   \max_{Z\succeq 0,y} \min_{X} \frac{1}{2}\|X-W\|^2+\langle y,b-\cA X\rangle-\langle Z,X\rangle \\
&=\max\limits_{Z\succeq 0,y,X}\left\{ \frac{1}{2}\|X-W\|^2+\langle y,b-\cA X\rangle-\langle Z,X\rangle: {\nabla_X} L(X,y,Z) = 0 \right\}\\
&= \max_{Z\succeq 0,y} -\frac{1}{2}\|Z+\cA^*y\|^2+\langle y,b-\cA W
\rangle-\langle Z, W\rangle.
\end{aligned}
\end{equation*}

\Cref{it:thduality-3}: Let $\bar X$ be the unique optimal solution,
as found by the above.
Then strong duality holds if, and only if, there exists
$(\bar{y},\bar{Z})$ such that the following \textdef{Karush-Kuhn-Tucker, \KKTp} conditions hold:
\begin{equation}
\label{eq:KKTprobone}
\begin{array}{cll}
\bar X-W-\cA^*\bar y- \bar Z= 0, & \bar Z\succeq 0, & \text{(dual feasibility),}
\\ \cA \bar X-b = 0, & \bar X\succeq 0, & \text{(primal feasibility),}
\\ \langle \bar Z, \bar X \rangle= 0, &  & \text{(complementary slackness).}
\end{array}
\end{equation}
Note that the complementary slackness condition and the fact that $\bar
X,\bar Z\succeq 0$ yield
\begin{equation}\label{eq:MoreauDe}
P_{\Snp}(W+\cA^*\bar y) = \bar X \text{ and } P_{\Snn}(W+\cA^* \bar y)
    = -\bar Z,
\end{equation}
due to $\bar X+(-\bar Z)=W+ \cA^*\bar y$ being the Moreau decomposition.
Finally, substituting $\bar X = P_{\Snp}(W+\cA^* \bar y)$ in
the primal feasibility condition, we conclude that the \KKT conditions
imply $F(\bar y) = \cA P_{\Snp}(W+\cA^* \bar y)-b = 0$.

Conversely, directly from the Moreau decomposition theorem, if
we are given some $\bar y \in Y$ satisfying $F(\bar y)=0$, then the tuple
$(\bar X, \bar y, \bar Z)$, with $\bar X$ and $\bar Z$ defined as
in~\cref{eq:MoreauDe}, satisfies the above \KKT conditions.
\samepage
\end{proof}

\Cref{th:duality}~\ref{it:thduality-3} shows how to obtain a solution to the primal problem~\Cref{eq:DNNparam} from a root  of $F$. In addition, the pair $(\bar y, \bar Z)$, with
\[
 \bar Z =  -P_{\Snn}(W+\cA^* \bar y),
\]  
is a dual optimal solution of the dual problem~\cref{eq:zerogap}.
This fact immediately follows from the
proof of~\Cref{th:duality}~\ref{it:thduality-3}, where we showed that the
tuple $(\bar X, \bar y, \bar Z)$ satisfies the \KKT conditions of
problem~\Cref{eq:DNNparam}.

\subsection{Regularization for Strong Duality}

The projection problem \eqref{eq:DNNparam} always admits a solution 
given that the feasible set $\cF$ is nonempty.
If the dual \eqref{eq:zerogap} of \eqref{eq:DNNparam} has an optimal solution, 
one can verify that the system given by the function in \eqref{eq:def_of_F}, i.e.,
\[
F(y) =  \cA\left( P_{\Snp}(W +\cA^*y)\right) - b 
\]
has a root $y\in \Rm$. 
However, when \eqref{eq:def_of_F} does not have a root, 
then strong duality fails. We elaborate on this pathology further
in~\Cref{sec:Patho}.
One way to guarantee that dual optimal set is nonempty is to
\emph{regularize} \eqref{eq:DNNparam} using \FRp. 
In \Cref{th:dualityFR} below, we list some properties induced by \FR  
properties that guarantee strong duality.

\index{minimal face of $\cF$, $f = \face(\cF,\Snp)$}
\begin{theorem}
\label{th:dualityFR}
Consider the projection problem \cref{eq:DNNparam} with data
$W,\cA,b$. Denote
\textdef{$f := \face(\cF,\Snp)$, the minimal face of $\Snp$ containing~$\cF$}. 
Let $\hat X \in
\relint f$ and let $V$ be a full column rank facial range vector with orthonormal columns with $\range V = \range \hat X$, $\rank(V)=r$.
Let 
\textdef{$\widebar W = V^TWV\in \Sr$}. Define the linear transformation  
\[
\textdef{$\cV (R) := VRV^T$}, \, R\in \Sr.
\]
Let $\bar \cA,\bar b$ define the affine constraints obtained
from  $(\cA \circ \cV)(\cdot),b$ after deleting
redundant constraints. Then the following hold:
\begin{enumerate}[(i)]
\item
\label{item:faceRestProb}
A facially reduced problem of \cref{eq:DNNparam} in the original
space $\Sn$ is
\begin{equation}
\label{eq:DNNparamFRf}
\begin{array}{rcl}
X^*(W) :=   & \argmin & \frac 12 \|X- W\|^2
\vspace{.05in}
\\ & \text{s.t.} &  \cA X = b, \, X \in f \qquad (X\succeq_f 0,\,f\unlhd \Snp).
\end{array}
\end{equation}
The \KKT conditions hold at $X^*(W)$ with optimal dual pair $y^*\in \Rm,
Z^*\in f^+$.

\item 
\label{item:faceRedProb}
A facially reduced problem of \cref{eq:DNNparam} 
in the smaller space $\Sr$ with surjective constraint 
$\bar \cA: \Srp \to \R^{\bar m}$ is
\begin{equation}
\label{eq:DNNparamFR}
\begin{array}{c}
\cV^\dagger (X^*(W))
   =R^*(\widebar W) 
:= \argmin \left\{ \frac 12 \|R-\widebar W\|^2
                 \,:\,  \bar \cA R = \bar b, \, R \in \Srp \right\},
\end{array}
\end{equation}%
\linelabel{line:MPinv} \label{pg:MPinv}%
where we denote $\cV^\dagger$ for the Moore-Penrose generalized
inverse of $\cV$.\footnote{The Moore-Penrose generalized inverse of
a linear transformation $\cV$ is the
unique linear transformation satisfying the four Penrose equations,
e.g.,~\cite{BiGr:74}.}
The \KKT conditions hold at $R^*(\widebar W)$ with optimal dual pair $y^*\in 
\R^{\bar m}, Z^*\in \Srp$.

\index{Moore-Penrose generalized inverse, $\cdot^\dagger$}
\index{$\cdot^\dagger$, Moore-Penrose generalized inverse}

\item
\label{item:strDualityFR}
\underline{Strong duality} holds for the \FR problems
\cref{eq:DNNparamFRf} and \cref{eq:DNNparamFR}. Moreover:
\[
\textdef{$X^* = \cV(R^*(\widebar W))$}
\text{   solves the original problem \cref{eq:DNNparam}};
\]
and, $\widebar R = R^*(\widebar W)$ is a solution to
the \FR primal problem \cref{eq:DNNparamFR} if, and only if,
\[
\widebar R = P_{\Srp}(\widebar W+\bar \cA^* \bar y),
\]
where  $\bar y$ is a root of the function 
\begin{equation}
\label{eq:Fbary}
\textdef{$\widebar F( y): = \bar \cA P_{\Srp}(\widebar W+\bar \cA^*  y) - \bar b$}.
\end{equation}
Equivalently, $X^*(W)$ is a solution to
the original primal problem \cref{eq:DNNparam} if, and only if,
there exists $\hat y$ such that
\begin{equation}
\label{eq:Ffy}
0=\textdef{$F_f( \hat y): = \cA P_f( W+ \cA^*  \hat y) -  b$}, \quad
X^*(W) = P_f(W+\cA^* \hat y),
\end{equation}
where $P_f$ stands for the projection onto $f := \face(\cF,\Snp)$
given by:
\begin{equation}
\label{eq:elegprojformorig}
P_f(u) = V\left(P_{\Srp}(V^T(u)V)\right)V^T  \quad \left(=
\cV \left(P_{\Srp} \cV^*(u)\right), \quad \cV^*\cV = I\right).
\end{equation}
\linelabel{line:defPf} \label{pg:defPf}
\end{enumerate}
\end{theorem}

\begin{proof}
The proof for
\Cref{item:faceRestProb,item:faceRedProb}
follows from the regularization in~\cite{bw3} with the
substitution $X \leftarrow VRV^T$. We note that the number of variables in the objective function reduces since $V$ has orthonormal columns and the norm is orthogonally
invariant.
The details follow from the proof of \Cref{th:duality} using the 
\KKT conditions after \FRp.
Note that the first-order 
optimality conditions for the facially reduced problem are:
\begin{equation}
\label{eq:KKTproboneFR}
\begin{array}{cll}
X-W-\cA^*y- Z= 0, & Z\succeq_{f^+}0, & \text{(dual feasibility),}
\\ \cA X-b = 0, & X\succeq_{f}0, & \text{(primal feasibility),}
\\ \langle Z, X \rangle= 0, &  & \text{(complementary slackness).}
\end{array}
\end{equation}

\index{\KKTp, Karush-Kuhn-Tucker}

\noindent
\Cref{item:strDualityFR}:
We first show the \emph{elegant} projection formula
\cref{eq:elegprojformorig}.
To show that the expression for $P_f(u)$ solves the nearest
point problem defined as $P_f(u) = \arg \min_{v\in f} \frac 12 \|v-u\|^2$,
we now verify the optimality conditions \cite[Theorem
3.1.1.]{MR1865628}:
\[
\trace \left\{(P_f(u)-u)(x-P_f(u))\right\} \geq 0, \, \forall x \in f,
\]
i.e., for each $x\in f$, there is $R\in\Srp$ such that $x = VRV^T$ and thus,
\[
\begin{array}{rcl}
	&& \trace\left\{( V(P_{\Srp}(V^T(u)V))V^T -u) (x- V(P_{\Srp}(V^T(u)V))V^T)
	\right\}
	\\&=&\trace\left\{( V(P_{\Srp}(V^T(u)V))V^T -u) (VRV^T- V(P_{\Srp}(V^T(u)V))V^T)
	\right\}
	\\&=&
	\trace\left\{ 
	V(P_{\Srp}(V^T(u)V))V^T VRV^T
	+uV(P_{\Srp}(V^T(u)V))V^T 
	\right. 
	\\ && \left.\quad -V(P_{\Srp}(V^T(u)V))V^T V(P_{\Srp}(V^T(u)V))V^T -uVRV^T
	\right\}
	\\&=&
	\trace\left\{ 
	(P_{\Srp}(V^T(u)V))(V^T VRV^TV)
	+(V^T uV)(P_{\Srp}(V^T(u)V))
	\right. 
	\\ && \left.\quad -(P_{\Srp}(V^T(u)V))(P_{\Srp}(V^T(u)V)) -uVRV^T
	\right\}
	\\&=&
	\trace\left\{ 
	(P_{\Srp}(V^T(u)V))(R)
	+(V^T uV)(P_{\Srp}(V^T(u)V))
	\right. 
	\\ && \left.\quad -(P_{\Srp}(V^T(u)V))(P_{\Srp}(V^T(u)V)) -V^TuVR
	\right\}
	\\&=&
	\trace\left\{ (P_{\Srp}(V^TuV) - (V^TuV))(R- P_{\Srp}(V^TuV))\right\} 
	\geq 0,
\end{array}
\]
where the last inequality comes from the projection.
This completes the proof of~\cref{eq:elegprojformorig} and establishes the
formula for the projection onto the face.

We continue to study the case where the \CQ (strict feasibility) fails.
With $P$ being the projection that makes the linear transformation onto,
we get \cref{eq:Fbary} is equivalent to: 
\[
\begin{array}{rcl}
\widebar F( y) 
&=&
 \bar \cA P_{\Srp}(\widebar W+\bar \cA^*  y) - \bar b
\\&=&
 (P \cA\circ \cV) P_{\Srp}( V^T WV +(P \cA\circ \cV) ^*  y) - P b,
\text{  for given data $W$}
\\&=&
 (P \cA\circ \cV) P_{\Srp}( V^T WV +(\cV^*\circ \cA^* P^T)(y) ) - P b
\\&=&
 (P \cA)\left(V\left[ P_{\Srp}( V^T (W +\cA^*P^Ty)V )\right]V^T\right) - P b
\\&=&
 (P \cA)\left( P_f(W +\cA^*P^Ty)\right) - P b
\\&=&
 (P \cA)\left( P_f(W +(P\cA)^*y)\right) - P b,
\end{array}
\]
where we have used the elegant formula~\cref{eq:elegprojformorig}.

This shows that we can work in the original space if we have done facial
reduction. Moreover,
\[
  P_f(W +\cA^*P^Ty) = 
 \cV\left[ P_{\Srp}( \cV^* (W +\cA^*P^Ty) )\right].
\]
Recall that $V^TV=I$.
In summary, necessity of~\cref{eq:Fbary} is clear. Therefore necessity
of~\cref{eq:Ffy} follows from
\[
\begin{array}{rcl}
 0 
&=&
\bar \cA P_{\Srp}(\widebar W+\bar \cA^*  y) - \bar b
\\&=&
(P\cA) \cV P_{\Srp}(\cV^*(W+ \cA^*P^T  y)) - P b
\\&=&
 (P \cA) P_f(W +\cA^*P^Ty) - P b.
\end{array}
\]
We can remove $P$ in the last line and ignore the redundant constraints.
\samepage
\end{proof}

We emphasize that strong duality holds in \Cref{th:duality}, 
\Cref{it:thduality-3}, only when there is a dual optimal solution that 
attains the dual optimal value $d^*$. However, 
strong duality is \emph{guaranteed} to hold 
in  \Cref{th:dualityFR}, \Cref{item:strDualityFR}, 
as a result of the regularization. 

\begin{remark}
\label{rem:projformula}
The proof of~\Cref{th:dualityFR} above provides the following
elegant formula for the projection of $u\in \Sn$
 onto the face $f=V\Srp V^T, V^TV=I$,
\begin{equation}
\label{eq:projfaceVSrVt}
\fbox{
$P_f(u) = V\left(P_{\Srp}(V^T(u)V)\right)V^T =
\cV \left(P_{\Srp} \cV^*(u)\right), \quad \cV^*\cV = I$,
}
\end{equation}
i.e.,~the work of finding the projection onto the face $f$ is reduced to the
well-known projection onto the smaller dimensional proper cone $\Srp$.
\end{remark}

We now consider dual optimal sets
\begin{equation}
\label{eq:solutionSets}
\cS := \{ y\in \Rm : F(y) = 0\} \text{ and } 
\cS_f := \{y\in \Rm : F_f(y)=0\}, 
\end{equation}
where $F$ is defined in \eqref{eq:def_of_F} and $F_f$ is defined in \cref{eq:Ffy}. 
In fact, when strict feasibility fails, $\cS$ is unbounded given that $\cS \ne \emptyset$ (see \Cref{thm:auxvectorNulFp} \ref{item:44-1} below).
Moreover, $\cS$ and $\cS_f$ are convex  and  $\cS \subseteq  \cS_f$.

\begin{theorem}
\label{thm:maxsdSolutions}
The facially reduced problem \cref{eq:Ffy} admits at least
$\maxsd(\cF)$ number of affinely independent dual solutions $y$. 
\end{theorem}
\begin{proof}
Let $(\bar{X}, \bar{y}, \bar{Z})$ be a triple satisfying \eqref{eq:KKTprobone}.
Let  $\lambda^1, \lambda^2, \ldots, \lambda^{\maxsd(\cF)}$ be vectors generated by \FR iterations. 
\linelabel{line:addingLamb} \label{pg:addingLamb}
By \eqref{eq:AuxSys}, each $\lambda^i$ satisfies $\cA^* (\lambda^i) \succeq 0$ and $\<\bar{X}, \cA^*(\lambda^i) \>=0$.  Since $\bar{y}$ satisfies $\bar X-W-\cA^*\bar y- \bar Z= 0$, it follows that $\bar X-W-\cA^*(\bar y - \lambda^i)- ( \bar Z + \cA^*(\lambda^i)) =0$. Moreover, 
$\bar Z + \cA^*(\lambda^i)\succeq 0$ and  $\<\bar{X}, \bar Z + \cA^*(\lambda^i) \>=0$.
Thus the vectors in the following set 
\[
\textdef{$\cS_\lambda$} := \bar{y} - \{\lambda^1, \ldots , \lambda^{\maxsd(\cF)}\}
=
\{
\bar{y} - \lambda^1, \ldots, \bar{y} - \lambda^{\maxsd(\cF)}
\}
\]
are solutions to \cref{eq:Ffy} as well.  Consequently, the result 
follows from the linear independence in~\Cref{prop:FRsequence}.
\samepage
\end{proof}

We show in \Cref{example:DiffSolSet} that $\cS$ and $\cS_f$ can differ, i.e., the containment $\cS \subseteq  \cS_f$ can be strict.

\begin{example}[$\mathcal{S} \subsetneq {\mathcal{S}_f}$]
\label{example:DiffSolSet}
Consider the instance $\cF$ given by the data:
\[
A_1 = \begin{bmatrix} 1&0&0\\ 0&0&0 \\ 0&0&0 \end{bmatrix}, \,
A_2 = \begin{bmatrix} 0&0&1\\ 0&1&0 \\ 1&0&0 \end{bmatrix}, \,
A_3 = \begin{bmatrix} 0&0&0\\ 0&0&0 \\ 0&0&1 \end{bmatrix},\,  \text{ and }
b = \begin{pmatrix} 1 \\ 0 \\ 0 \end{pmatrix}.
\]
The singularity degree of $\cF$ is $2$, $\sd(\cF) =2$.
The first \FR iteration yields a face that strictly contains the minimal
face and corresponds to
$\lambda^1 = \begin{pmatrix} 0 \\ 0 \\ 1 \end{pmatrix}$ with the facial range vector 
$V_1 = \begin{bmatrix} 1&0 \\ 0&1\\ 0&0 \end{bmatrix} $;
and the second \FR iteration yields 
$\lambda^2 = \begin{pmatrix} 0 \\ 1\\ 0 \end{pmatrix}$ with the facial range vector 
$V_2 = \begin{pmatrix} 1 \\ 0\end{pmatrix}$. 
Thus, the minimal facial range vector $V$ for $\cF$ is $V = V_1V_2 = \begin{bmatrix} 1 \\0 \\ 0 \end{bmatrix} $.
The facially reduced system is 
$\{R \in \Sc_+^1: 
\begin{bmatrix} 1 \end{bmatrix} R = 1 \}$. 
\linelabel{line:singletonF} \label{pg:singletonF}
We note that $\cF$ contains a unique point~$e_1e_1^T$.

We now consider the \BAP \cref{eq:DNNparam} with 
$W = \begin{bmatrix}0&0&0\\0&-1&-1\\0&-1&0\end{bmatrix} $. 
We consider the triple $(\bar{X},\bar{Z}, \bar{y})$ where
\[
\bar{X} = \begin{bmatrix} 1&0&0\\ 0&0&0 \\ 0&0&0 \end{bmatrix},
\bar{Z} = \begin{bmatrix} 0&0&0\\ 0&1&1 \\ 0&1&2 \end{bmatrix}, \text{ and }
\bar{y} = \begin{pmatrix} 1 \cr 0 \cr -2 \end{pmatrix} .
\]
The triple $(\bar{X},\bar{Z},\bar{y})$ satisfies the first-order optimality conditions.

We note that 
$\bar{y} - \lambda^1$ and $\bar{y} - \lambda^2$ are solutions to \cref{eq:Ffy}.
However,  $\bar{y} - \lambda^2$ is not a solution to \cref{eq:def_of_F} since 
\[
W+\cA^*(\bar{y} - \lambda^2) 
=
\begin{bmatrix}0&0&0\\0&-1&-1\\0&-1&0\end{bmatrix}  
+ \begin{bmatrix}1&0&1\\0&1&0\\1&0&-2\end{bmatrix} 
= 
\begin{bmatrix} 1&0&-1\\ 0&-2&-1 \\ -1&-1&-2 \end{bmatrix},
\]
and 
\linelabel{line:fix_sign}
$\bar{X} - W-\cA^*(\bar{y} - \lambda^2)$ has a negative eigenvalue. 
\end{example}

It is of interest that the containment relation
$\mathcal{S} \subsetneq {\mathcal{S}_f}$ in \Cref{example:DiffSolSet}
stems from the solutions to \eqref{eq:AuxSys}.
This can also be seen from the optimality characterization
\eqref{eq:KKTprobone} with $Z\succeq 0$ and the regularized
characterization~\eqref{eq:KKTproboneFR}
with $Z\succeq_{f^+} 0$, where the failure of strict feasibility 
implies $f\subsetneq \Snp, f^+\supsetneq \Snp$.

\section{A Basic Newton Method}
\label{sec:basicNewton}

\label{sect:optcond}

We design a Newton-like method that solves for a root
$\bar y$, $F(\bar y)=0$, where
\[
F(y) = \cA P_{\Snp}(W+\cA^*y)-b. 
\]
Rather than applying an optimization
algorithm to solve the dual as in~\cite{MR2112861}, we highlight that
we solve a system of equations of the form $F(y)=0$ to find points satisfying the first-order optimality
conditions as is done in \cite{BoWo:86,MiSmSwWa:85,CensorMoursiWeamsWolk:22}.
Given that $F(\bar{y})=0$, it follows that the optimum of the $\BAP$ \eqref{eq:DNNparam} is $\bar X =  P_{\Snp}(W+\cA^* \bar y)$.
We note that $P_{\Snp}$ is found using the Eckart-Young Theorem
\cite{EckartYoung:36},
i.e.,~we use a spectral decomposition and set the negative eigenvalues
to $0$.
Primal feasibility is immediate from the definitions and the projection.
An application of the Moreau theorem yields the dual feasibility and
complementarity.

We now  present the pseudo-code of our Semi-Smooth Newton Method for the
\BAP \Cref{eq:DNNparam} in \Cref{alg:1}.
The algorithm proceeds as follows: at each iteration, the search direction is computed by forming a (Clarke generalized) Jacobian $J_k$ of $F$ at the current iterate $y_k$, and the triple $(X_k,Z_k,y_k)$ is then updated using this search direction.  We resort to evaluations of unit vectors to construct $J_k$, the  computational steps are discussed in~\Cref{sect:ADDF}

\begin{algorithm}[h]
	\caption{Semi-Smooth Newton Method for \BAP for Spectrahedra}
	\label{alg:1}
	\begin{algorithmic}[1]
		\Require{
$\left(W\in\Sn, \cA:\Sn\to \Rm, b\in \Rm\right)$, 
$\left(y_0\in\Rm, \varepsilon>0, \text{ maxiter} \in \mathbb{N}\right)$}
		\State{\textbf{Initialization:} $k\leftarrow 0$, $F_0 \leftarrow \cA(X_0)-b$, 
		      stopcrit $\leftarrow  \|F_0\|/(1+\|b\|)$}, \\ \hspace{2.7cm} $X_0
		\leftarrow \PSnp(W+\cA^* y_0)$, $Z_0 \leftarrow
		(X_0-W-\cA^*y_0)$
		\While{(stopcrit  $>\varepsilon$) \& ($k\leq$ maxiter)}
		\State{construct Jacobian $J_k$ by evaluating at unit vectors
		     $J_k(e_i)$} 
		\label{line:algoJkReg}
		\State{choose a regularization parameter $\sigma \geq 0$
		for positive definite $\bar{J} = (J_k+\sigma I_m)$}
		\State{solve system $\bar{J} d = -F_k$} \Comment{(Obtain
		Newton direction $d$)} 
		\State{\bf update:}
		\State{$\qquad y_{k+1}\leftarrow y_k+d$}
		\State{$\qquad X_{k+1}\leftarrow \PSnp(W+\cA^*y_{k+1})$}
		\State{$\qquad Z_{k+1}\leftarrow  X_{k+1}-(W+\cA^* y_{k+1})$ }
		\State{$\qquad F_{k+1} \leftarrow \cA (X_{k+1}) -b$}
		\State{$\qquad \text{stopcrit } \leftarrow 
		           \|F_{k+1}\|/(1+\|b\|)$}
		\State{$\qquad k\leftarrow k+1$}
		\EndWhile
		\State{\textbf{Output:} Primal-dual near optimum: $X_k, (y_k,Z_k)$}
	\end{algorithmic}
\end{algorithm}

\begin{remark}[On the convergence of \Cref{alg:1}]\label{remark:convAlg}
Convergence guarantees for \Cref{alg:1} are derived from  the so-called \emph{inexact Semismooth Newton's Algorithm} in  Facchinei and Pang's book; see \cite[Algorithm~7.5.4]{facchinei2003finiteVol2}. 
Specifically, Facchinei and Pang consider an inexact Newton's method for obtaining a zero of a semismooth operator $F:\mathbb{R}^{m}\to\mathbb{R}^m$, with search direction $d^k$  defined by
\[
F(x^k) + J_k d^k = r^k, \quad \forall k,
\]
where $J_k$ is an element of the Clarke generalized Jacobian of $F$ at $x^k$, and the error $r^k$ satisfies
\[
\|r^k\| \leq	 \eta_k  \|F(x^k)\|
\]
for some nonnegative sequence $\{\eta_k\}$.
According to \cite[Theorem~7.5.5]{facchinei2003finiteVol2}, if the generalized Jacobian is nonsingular at a point $x^*$ satisfying $F(x^*)=0$, and if there exists a constant $\bar{\eta}>0$ such that $\eta_k \leq \bar{\eta}$ for all $k$, then the method converges locally to $x^*$.

Our algorithm can be regarded as a realization of this inexact Semismooth Algorithm  where  $r^k = -\sigma d^k$ , with $\sigma>0$, and therefore $d^k = - (J_k + \sigma I_m)^{-1} F(x^k)$. Hence,  we may choose $\eta_k$ as
\[
\eta_k =  \frac{\|r^k\|}{\|F(x^k)\|} =  \frac{ \|( J_k + \sigma I_m)^{-1} F(x^k)  \|}{\|F(x^k)\|}\sigma   \leq  \frac{\sigma} {\lambda_{\min} (J_k)  + \sigma} \leq 1, \quad \forall k,
\]
where $\lambda_{min}(J_k)$ stands for the smallest eigenvalue of $J_k$.
Therefore,  \cite[Theorem~7.5.5]{facchinei2003finiteVol2} provides local convergence guarantees for \Cref{alg:1}.
\end{remark}

\subsection{Alternative Directional Derivative Formulation}
\label{sect:ADDF}

We now outline the computation of the Jacobian $J_k$
at line \ref{line:algoJkReg} in \Cref{alg:1}. 
In principle, obtaining a Clarke generalized Jacobian of $F$ in our semismooth Newton method would
require computing an element in the Clarke generalized Jacobian of  $P_{\Snp}$.
Every element in the generalized Jacobian 
is a $4$-\emph{tensor} on $\Rn$, whose complete
formulation can be found in~\cite{MR2252652}.  In matrix form this would be
expressed as a square matrix of order $n^4$. The memory requirements for
storing a matrix of such dimension can be too demanding even for
moderate values of $n$. In particular, \textsc{Matlab} software would
have problems with size $n\geq150$.

In order to overcome the memory deficiency,  
we make use of an elegant formula for evaluating Clarke generalized Jacobians
of $P_{\Snp}$ derived by Sun in~\cite{Su:06}, which builds upon earlier developments in~\cite{PangSunSun:03,sun2002semismooth}. In what follows we show this formula can be leveraged to efficiently compute a Clarke generalized Jacobian of $F$ through evaluations involving only the unit vectors $e_i\in \Rm$. We begin by introducing the following notation.

\linelabel{line:updatedJac_begin}
Let $S = U\Lambda U^T\in \Sn$ be the spectral decomposition with vector of eigenvalues $\lambda$ given as above. 
We partition the index sets of the eigenvalues based on their signs:
\[
\alpha = \{i:\lambda_i>0\},\,
\beta = \{i:\lambda_i=0\},\,
\gamma = \{i:\lambda_i<0\}. \]
Accordingly, the eigenvalue and eigenvector matrices can be block-partitioned  as
\[
\Lambda = \blkdiag(\Lambda_\alpha, 0 ,
\Lambda_\gamma), \ \text{ and }\ U=[U_\alpha\,U_\beta\,U_\gamma].
\]
Next, we define $\Omega \in \Sn$ entry-wise by
\begin{equation}
\label{eq:defOmega}
\Omega_{ij} =
\frac
{\max(\lambda_i,0)+\max(\lambda_j,0)}{|\lambda_i|+|\lambda_j|},\,\forall
i,j,
\end{equation}
where the convention $\frac 00:=1$ is adopted. 
By \cite[Proposition~2.2]{Su:06}, for any matrix $S=U\Lambda U^T$, a linear map $G_1$ belongs to the Clarke generalized Jacobian $\partial_C  P_{\Snp}(S)$  if and only if there exists $G_2 \in \partial_C P_{\mathbb{S}_+^{|\beta|}}(0)$ such that for all $H \in \Sn$,
\begin{equation}
\label{eq:G1atH}
G_1(H)=
U\begin{bmatrix}
\tilde H_{\alpha \alpha}  & \tilde H_{\alpha \beta}  &
             \Omega_{\alpha\gamma}\circ \tilde H_{\alpha \gamma}  \cr
\tilde H_{\alpha \beta}^T  & G_2( \tilde H_{\beta
                          \beta})  & 0 \cr
\tilde H_{\alpha \gamma}^T\circ \Omega_{\alpha \gamma}^T  &  0 & 0
\end{bmatrix}U^T,   
\end{equation}
where $\tilde H = U^THU$. 
Equation \eqref{eq:G1atH} indicates that any element of the Clarke generalized Jacobian $\partial_C  P_{\Snp}(S)$  is determined by a Clarke generalized Jacobian  in  $G_2 \in \partial_C P_{\mathbb{S}_+^{|\beta|}}(0)$. Consequently, the problem reduces to finding an element in  $\partial_C P_{\mathbb{S}_+^{|\beta|}}(0)$.

In particular, by~\cite[Example~3.8]{MR2252652}, the Clarke subdifferential of $ P_{\mathbb{S}_+^{|\beta|}}$ at $0$ is characterized by the set of $4$-tensors  on $\mathbb{R}^{|\beta|}$ given by
\[
\partial_C P_{\mathbb{S}_+^{|\beta|}}(0) = \mathrm{conv}  \left( \mathcal{O}^{|\beta|} \cdot (\Diag^{(12)} \mathcal{D}_{\{01\}} (|\beta|) ) \right),
\]
where  $\mathrm{conv}(S)$ stands for the convex hull of a set $S$ and
\begin{enumerate}[(i)]
\item $\mathcal{D}_{\{01\}}(|\beta|)$ is the set of $|\beta|\times|\beta|$ symmetric matrices with entries in $\{0,1\}$ such that the entries in each row (from left to right) or column (from top to bottom) form a noincreasing sequence;
\item for any matrix $M$ in $\R^{|\beta| \times |\beta|}$,  $\Diag^{(12)} M$ is defined as the 4-tensor on $\R^{|\beta|}$ whose components are given by 
\[
(\Diag^{(12)} M)_{\scriptsize\begin{array}{c} i_1 i_2 \\ j_1 j_2\end{array}} =  
\left\{
\begin{array}{ll}
M_{i_1i_2},  &  \text{if } i_1 = j_2 \text{ and } i_2 = j_1, \\
0, & \text{otherwise}.
\end{array}
\right.
\]
\end{enumerate}
We note that, since the zero matrix of size $|\beta|\times|\beta|$ belongs to  $D_{\{01\}}(|\beta|)$, the zero tensor from $\mathbb{R}^{|\beta| \times |\beta|}$ to $\mathbb{R}^{|\beta| \times |\beta|}$ belongs to $\partial_C P_{\mathbb{S}_+^{|\beta|}}(0)$.

Hence,  in \Cref{alg:1}, we can always choose the generalized Jacobian $G_1\in \partial_C  P_{\Snp}(S)$  associated with $G_2=0$.  Under this choice, the evaluation of $G_1$ at any $H\in\Sn$ simplifies to
\begin{equation}\label{eq:G1atH2}
G_1(H)=
U\begin{bmatrix}
\tilde H_{\alpha \alpha}  & \tilde H_{\alpha \beta}  &
             \Omega_{\alpha\gamma}\circ \tilde H_{\alpha \gamma}  \cr
\tilde H_{\alpha \beta}^T  & 0 & 0 \cr
\tilde H_{\alpha \gamma}^T\circ \Omega_{\alpha \gamma}^T  &  0 & 0
\end{bmatrix}U^T,   \quad \forall H \in \Sn,
\end{equation}
where $\tilde H = U^THU$. 
Finally, observe that a Clarke Jacobian for our function $F$ is a matrix that can be explicitly constructed by evaluating this operator along the standard basis vectors.
\linelabel{line:updatedJac_end}

The following lemma provides an explicit expression for evaluating this particular Jacobian $J \in \partial_C F(y)$ for any $y\in\mathbb{R}^n$ along any direction.

\begin{lemma}
\label{lemma:JacobianRepre}
Let $y \in\mathbb{R}^m$ be such that $Y:= W +\cA^*y \in
\Sn$.
Let $Y := U\Diag(\lambda(Y)) U^T$ be a spectral decomposition of $Y$
such that the eigenvalues $\lambda_1\geq \lambda_2 \geq \ldots \geq
\lambda_n$ are sorted in nonincreasing order, and denote with $\alpha$, $\beta$
and $\gamma$ the sets of indices associated with positive, zero and negative
eigenvalues, respectively, i.e.,~$\alpha:=\{ i: \lambda_i >0\}$, $\beta:=\{i:\lambda_i=0\}$ and $\gamma=\{i: \lambda_i <0\}$. 
Then there exists a Clarke generalized Jacobian $J$ in $\partial_C F$ at $y$ such that its evaluation at any direction $\Delta y\in\mathbb{R}^m$ is given by
\begin{equation}
\label{eq:JFeval}
J(\Delta y) = \cA \left(U
 \begin{bmatrix}
\tilde H_{\alpha \alpha}    & \tilde H_{\alpha\beta} & 
             \Omega_{\alpha\gamma}\circ \tilde H_{\alpha \gamma}  \cr 
             \tilde H^T_{\alpha \beta} & 0 & 0 \cr
\tilde H_{\alpha \gamma}^T\circ \Omega_{\alpha \gamma}^T   & 0 & 0 \cr
\end{bmatrix}U^T\right) ,
\end{equation}
where $\tilde{H}:= U^T\left( \cA^*\Delta y\right)U$.
\end{lemma}

\begin{proof}
It suffices to recall that the chain rule for Clarke subdifferential applied to $F$ yields
\[
\partial_C F(y) = \mathcal{A} \partial_C  P_{\Snp} (W + \mathcal{A}^*y ) \mathcal{A}^* ,
\]
where the equality above is by \cite[Theorem~2.3.10]{Clarke1990optimization}, since semismooth implies Clarke regular.
Hence, by taking $G_1\in \partial_C  P_{\Snp}(S) $ given  above for $S=W+\mathcal{A}^* y$, the equation in \Cref{eq:G1atH2} provides~\Cref{eq:JFeval}.
\end{proof}

We now outline the steps for computing the Clarke generalized Jacobian $J_k$
at line \ref{line:algoJkReg} in \Cref{alg:1}. Specifically, by \Cref{lemma:JacobianRepre}, the Jacobian of $F$ evaluated at $y\in\mathbb{R}^m$, $J(y)$, is computed following the steps below. 
\begin{enumerate}[(i)]
\item 
Let  $Y = W + \cA^* y \in \Sn$. Consider the spectral decomposition of $Y$ given by
\[
Y =
\begin{bmatrix} V_\alpha & V_\beta & V_\gamma \end{bmatrix} \Diag(\lambda (Y))
\begin{bmatrix} V_\alpha & V_\beta & V_\gamma \end{bmatrix}^T,
\]
where $V_\alpha, V_\beta$ and $V_\gamma$ are the matrices of  eigenvectors  associated with the positive, zero and negative eigenvalues of $Y$, respectively.

\item  
Define the rotation 
$\textdef{$\mathcal{R}_{Y}$} : \Sn \to \Sn$, by
\begin{equation}
\label{eq:Rrotat}
\mathcal{R}_{Y}(\rho)  :=
\begin{bmatrix} V_\alpha & V_\beta & V_\gamma \end{bmatrix} \rho \begin{bmatrix} V_\alpha & V_\beta & V_\gamma \end{bmatrix}^T ;
\end{equation}
\item For each $j=1,\ldots,m$, compute
\begin{equation}
\label{def:Ti}
T_j := \begin{bmatrix}
V_\alpha^TA_j V_\alpha &  V_\alpha ^T A_j V_\beta  & \Omega_{\alpha\gamma} \circ V_\alpha ^TA_j V_\gamma \\
(V_\alpha^T A_j V_\beta)^T & 0 & 0 \\
\bigl( \Omega_{\alpha\gamma} \circ  V_\alpha ^T A_j V_\gamma^T \bigr)^T & 0 & 0
\end{bmatrix} \in \Sn ;
\end{equation}
\item 
The $j$-th column of the Jacobian at $y$, $J(y)$, is
\begin{equation}
\label{eq:eachJacCol}
\cA(  \mathcal{R}_{Y}(T_j))
=:
A \svec \left(  \mathcal{R}_{Y}(T_j) \right).
\end{equation}
\end{enumerate}

\section{Failure of Regularity and Degeneracy}
\label{sec:FailureReg}

This section examines various aspects of \Cref{alg:1} that result from 
the failure of strict feasibility.
The failure of regularity is known to result in pathologies in conic
programs, both on theoretical and practical aspects. 
We show that \Cref{alg:1} is not an exception to this phenomenon. 

This section is organized in three parts. 
In \Cref{sec:Patho} we discuss two types of pathologies.
One well-known pathology is the possibility of failure of strong
duality. Since the primal and dual optimal values agree
(\Cref{th:duality} \ref{it:thduality-2}), the only difficulty left is that
the dual optimal value may not be attained by any dual feasible point.
We identify a condition where this occurs (see \Cref{lem:closure})
and show how to construct instances where strong duality fails.
Another well-known consequence of the absence of strict feasibility is that
the dual optimal set is unbounded \cite{MR0489903}. We explain why
\Cref{alg:1} experiences difficulties in this case in \Cref{sec:unboudedDualSet}.
The second part in
\Cref{sec:JacoDegen} is devoted to understanding the properties of the Jacobian of $F$
computed near the optimal point as seen through the lens of degeneracy.
We recall the discussions from \Cref{sec:degen} to help explain the
behaviour of \Cref{alg:1}.
In particular, we rely on the fact that \emph{every point in $\cF$ 
is degenerate in the absence of strict feasibility}. 
We conclude in \Cref{sec:degenofMCQAP}
with the application of degeneracy identification 
to two real-world examples: the elliptope and the vontope.

\index{\QAPp, quadratic assignment problem}
\index{quadratic assignment problem, \QAPp}

\subsection{Pathologies Arise in the Absence of Strict Feasibility}
\label{sec:Patho}

In this section we discuss pathologies that arise as a result 
of the absence of strict feasibility.
We provide a method of constructing instances 
where the dual optimal value is not attained.
In addition, assuming that the dual
optimal value is attained,  we provide members that certify the 
unbounded dual optimal set; and we examine
the behaviour of \Cref{alg:1}.

\subsubsection{Unattained Dual Optimal Value}
\label{sect:unattaineddual}
\index{$(x,y)$, open interval}
\index{open interval, $(x,y)$}
\index{$p^* = d^*$, zero duality gap}
\index{$\null A$, null space of $A$}

\Cref{th:duality} states that there
is always a \textdef{zero duality gap, $p^* = d^*$} and 
 the solution value of the primal problem, $p^*$, is attained.
However, in the absence of strict feasibility,
the dual attainment does not necessarily hold. 
\Cref{examp:KKTfailsS2} below illustrates that  
strong duality can fail for~\cref{eq:DNNparam} when strict feasibility
fails.

\begin{example}[Failure of strong duality]
\label{examp:KKTfailsS2}
Consider the following instance of \BAP~\cref{eq:DNNparam} given by
\begin{equation}
\label{eq:sdfailEx}
\min_X \left\{ \frac{1}{2}\Bigg\|X - \begin{bmatrix} 0 & -1 \cr -1 & 0
\end{bmatrix}\Bigg\|^2 \ : \ X_{11} = 0, \,\; X\succeq_{\Sc^2_+} 0 \right\}.
\end{equation}
The set of feasible solutions of \eqref{eq:sdfailEx} is $\{X \in \Sc^2 :  X_{11}=X_{12}=X_{21} = 0, X_{22} \geq 0 \}$.
Therefore, the optimal value of the problem is 
\begin{equation*}
1 = \min_{X_{22} \geq 0 } \frac{1}{2} \Bigg\| \begin{bmatrix} 0 & 1 \cr 1 & X_{22} \end{bmatrix} \Bigg\|^2 = \frac{1}{2} \left( 2 + X_{22}^2 \right),
\end{equation*}
which is attained when $X_{22} = 0$. In other words, the optimal solution of the best approximation problem  is attained at $ \bar X=0$. 

Now, note that the primal constraint in \eqref{eq:sdfailEx} is given by $\trace\left( E_{11}X\right) = \cA X = 0$, and therefore $\cA^* y = y E_{11}$ for all $y\in\mathbb{R}$. 
Thus, dual feasibility of the optimality conditions (see~\eqref{eq:KKTprobone})  implies 
\begin{equation*}
- \begin{bmatrix} 0 & -1 \cr -1 & 0 \end{bmatrix} - \bar{y} E_{11} =
\begin{bmatrix} - \bar{y} & 1 \cr 1 & 0 \end{bmatrix}\in \Sc^2_+, 
\text{ for some } \bar{y} \in \mathbb{R}.
\end{equation*} 
However this does not hold for any $\bar{y}\in\mathbb{R}$.  Thus attainment fails for the dual.
\end{example}

\Cref{examp:KKTfailsS2} above illustrates that strong duality may fail in the absence of strict feasibility; the linear manifold defined by $X_{11}=0$ entirely consists of singular matrices. 
We note that strong duality can hold even in the absence of strict feasibility. 
\Cref{rem:KKTfails} presents a constructive approach for generating instances that
fail strong duality. We first recall the following.
\begin{lemma}[\!\!{\cite[Lemma 2.2]{RaTuWo:95}}]
\label{lem:closure}
Suppose that $0\neq K\unlhd \Snp$, is a proper face of $\Snp$. Then
\[
	\Snp + K^\perp = \overline{\Snp + \spanl K^{\Delta}}.
\]
Furthermore, 
\begin{equation}
\label{eq:notclosed}
\Snp + \spanl K \text{  is not closed}.
\end{equation}
\end{lemma}

\begin{remark}[Constructing instances that fail strong duality]
\label{rem:KKTfails}
The dual feasibility of the first-order optimality conditions \cref{eq:KKTprobone} states:
\[
\bar X-W \in \range (\cA^*) + \Snp.
\]
From \cref{eq:notclosed}, we can choose any proper
face $K\unlhd \Snp$ and construct a linear map $\cA$ to satisfy
$\range(\cA^*) = \spanl K$.
Therefore, 
\[
\bar X - W \in
\overline{\range (\cA^*) + \Snp} \backslash
            \left(\range (\cA^*) + \Snp\right), \, \bar X\succeq 0,
\]
results in the failure  of \cref{eq:KKTprobone}. \Cref{examp:KKTfailsS2} indeed falls into this category.
Note that we can
always choose $b=\cA \bar X$ so that we still have a zero duality gap.
\end{remark}

\subsubsection{Unbounded Dual Optimal Set and Singular Jacobian}
\label{sec:unboudedDualSet}

We now discuss a property of the dual optimal set that, if it exists,
results in a poor behaviour of \Cref{alg:1}.
Recall that the absence of strict feasibility of $\cF$ implies the existence of a solution $\lambda$ of the auxiliary system~\Cref{eq:AuxSys}. 
We use the solution $\lambda$ of \cref{eq:AuxSys} to derive two properties  
of the dual solution set \textdef{$\cS = \{ y\in \Rm : F(y) = 0\}$}
defined in~\eqref{eq:solutionSets}:
\begin{equation}
\label{item:Sunbdd}
\begin{array}{rl}
\text{Slater fails, } \cS \neq \emptyset \implies & 
\left\{
\begin{array}{l}
\text{(i): solution set $\cS$ is unbounded;}
\\\text{(ii) Jacobian at any solution $\bar{y} \in \cS$ is singular.}
\end{array}
\right.
\end{array}
\end{equation}
\Cref{thm:auxvectorNulFp} below clarifies the conditions that result in  
the unbounded dual solution set in~\Cref{item:Sunbdd}(i); it then
explains why we get an ill-conditioned Jacobian and thus
provides a rationale for regularization of the search direction at 
line \ref{line:algoJkReg} in \Cref{alg:1}.

\begin{theorem}
\label{thm:auxvectorNulFp}
Suppose that strict feasibility fails for the (primal) spectrahedron
\cref{eq:DNNparam} \emph{but} strong duality holds. 
Let $\bar{y}\in\cS$ and let $\lambda$ be any solution to \cref{eq:AuxSys}.
Then the following hold:
\begin{enumerate}[(i)]
\item \label{item:44-1}
The solution set $\cS$ is unbounded. 
Moreover, $\lambda$ provides a \textdef{recession direction},
$F(\bar{y}-t\lambda) = 0, \forall t \in \Rp$.
\item \label{item:44-2} 
The directional derivative of $F$ at $\bar y$ along $\lambda$ exists and is equal to zero.
\item \label{item:44-3} 
In addition suppose that $F$ is differentiable at $\bar y\in\cS$. Then the Jacobian $F^\prime( \bar y)$ is singular.
Moreover, $\lambda \in \nul F^\prime( \bar y)$.
\end{enumerate}
\end{theorem}

\begin{proof}
\Cref{item:44-1}:
Let $(\bar X, \bar{y},\bar{Z})$ be a triple that satisfies the optimality conditions in \cref{eq:KKTprobone}.
We now let $\lambda$ be a solution to the auxiliary system \cref{eq:AuxSys} and $Z := \cA^*\lambda  \succeq 0$. 
We aim to show that, for any $t>0$, the triple $(\bar X, \bar{y}-t\lambda,\bar{Z}+t Z)$ also satisfies the optimality conditions. Indeed, for all $t>0$, we have
$\bar Z+ t Z \succeq 0$  and
\[
\begin{array}{rcl}
0 
&=&
\bar X-W-\cA^*\bar y- \bar Z
\\&=&
\bar X-W-\cA^*(\bar y-t \lambda)- (\bar Z+ t Z).
\end{array}
\]
The verification of primal feasibility is trivial. Finally complementarity follows:
\[
\langle \bar Z + t Z, \bar X \rangle = t \langle Z, \bar X \rangle = t \langle \lambda, \cA \bar X \rangle = t \langle \lambda, b \rangle = 0, \quad \forall t >0,
\]
where the last equality follows from~\Cref{eq:AuxSys}. Finally, by the proof of~\Cref{th:duality}~\ref{it:thduality-3} we conclude that $\bar{y} - t \lambda $ is a root of $F$ for all $t >0$, or equivalently,
\[
\{ \bar{y} - t \lambda : t \in \R_+ \} \subseteq \cS.
\]

\Cref{item:44-2}: This directly follows from the fact that $F(\bar y) = F(\bar y - t \lambda)$ for all $y\in\cS$ and $t\in\R_+$.

\Cref{item:44-3}: Suppose $F$ is differentiable at a point $\bar y \in \cS$. Then the partial derivative of $F$ at $\bar y$ in the direction of $\lambda$ is given by
\[
F'(\bar y)\lambda = 0,
\]
where $F'(\bar y)$ denotes the Jacobian of $F$ at $\bar y$. 
\samepage
\end{proof}

We note that the system \eqref{eq:AuxSys}  may contain multiple linearly independent solutions. Let $\{\lambda^1,\ldots, \lambda^k\}$ be a set of linearly independent solutions to \cref{eq:AuxSys}.
Hence by \Cref{thm:auxvectorNulFp} 
we deduce that 
the solution set $\cS$ contains a $k$-dimensional recession cone.
Moreover, 
if the differentiability of $F$ at $\bar{y}$ is further assumed, 
$\nul F^\prime(\bar{y})$ contains at least $k$ zero singular values. 
Another interesting consequence of \Cref{thm:auxvectorNulFp} is that 
if $F^\prime(\bar{y})$ is nonsingular, then strict feasibility holds for $\cF$.

The unboundedness of the set $\cS$ immediately translates into the
unboundedness of the set of optimal solutions of the dual
problem~\Cref{eq:zerogap}. The proof of~\Cref{thm:auxvectorNulFp}~\Cref{item:44-1} shows that the triple $(\bar X, \bar y - t\lambda, \bar Z + t \cA^*\lambda)$ satisfies the optimality conditions~\Cref{eq:KKTprobone}  for all $t\in\R_+$. 
Therefore, the unbounded set
\[
\{ (\bar y, \bar Z) + t(-\lambda, \cA^* \lambda) \, : \, t\in\R_+\}
\]
constitutes recession directions of the set of dual solutions.

Furthermore, 
as stated in \cite{MR0489903} (and related papers), the dual optimal
set is unbounded when Mangasarian-Fromovitz type constraint
qualifications fail. Here this means that 
the dual optimal set is unbounded 
when strict feasibility fails for the primal~\cref{eq:DNNparam}, since
we assume that $\cA$ is surjective.
Our dual~\Cref{alg:1} can encounter numerical difficulties in this
setting.
We now provide the details why the norms of the dual
variables typically diverge.

Let $\bar{y} \in \cS$.
Suppose that we are at a point $\hat{y}$ such that 
$0\neq \epsilon = F(\hat{y}),\, \hat{y} = \bar{y} + \Delta y$.
We note that 
\[
\epsilon = F(\bar{y} + \Delta y)  - F(\bar{y}) \approx F^\prime
(\bar{y}) \Delta y.
\]
When $||\epsilon||$ is small, $\Delta y$ is close to being in $\nul (F^\prime (\bar{y}))$.
We have shown in \Cref{thm:auxvectorNulFp} that 
a solution $\lambda$ to \cref{eq:AuxSys} always satisfies 
\[
F(\bar{y} + \lambda) = 0 \text{ and } F^\prime (\bar{y}) \lambda =0.
\]
Therefore, \emph{large} choices for $\Delta y$ 
that include a large component from the nullspace is possible. This is
in particular true when the Jacobian is singular at the optimum and the
regularization parameter is converging to zero.

A typical behaviour of \Cref{alg:1} in the absence of strict feasibility
is illustrated in~\Cref{fig:ConvBehaviour}, i.e.,~we see the growth of
the norms of the dual variables.
\begin{figure}[h]
\centering
\includegraphics[height=5cm]{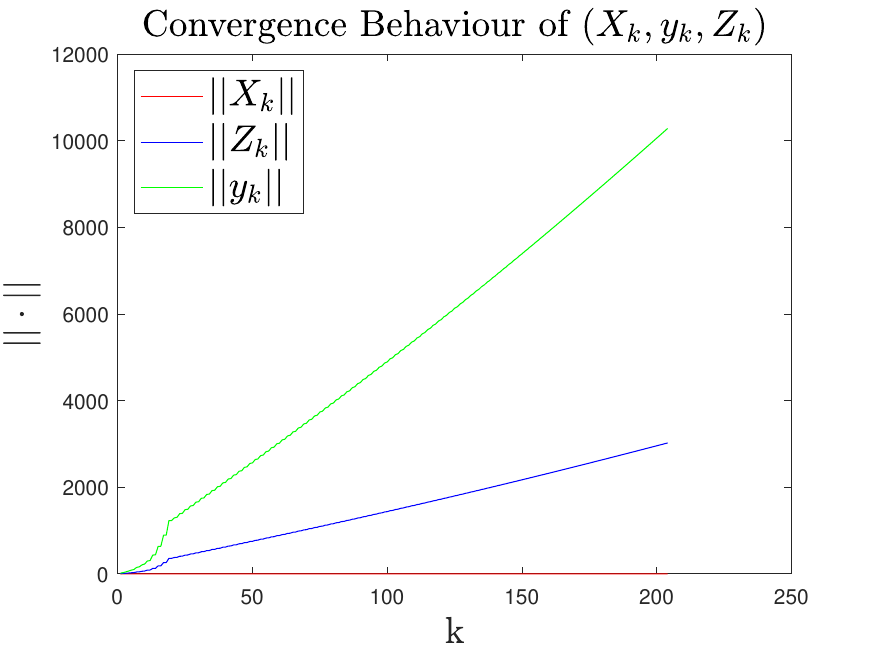}
\caption{Typical behaviour of $\{(X_k,y_k,Z_k)\}$ from~\Cref{alg:1}  in the absence of strict feasibility.}
\label{fig:ConvBehaviour}
\end{figure}

\subsection{Jacobian Behaviour Near-Optimum and Degeneracy} 
\label{sec:JacoDegen}

In this section we study properties of the Jacobian of $F$ computed near the optimal point 
and relate its behaviour to the degeneracy status of the optimal point.
We show that the degeneracy status of the
optimal point \emph{characterizes} the singularity of the Jacobian matrix.

We extend the discussion of computing the Jacobian presented in  \Cref{lemma:JacobianRepre} and elaborate the computational steps. 
Let $(\bar{X}, \bar{y}, \bar{Z})$ be an optimal triple that solves 
\cref{eq:KKTprobone}.
We further assume that $\bar{X}$ and $\bar{Z}$ satisfy \emph{strict complementarity}. 
\linelabel{line:strictCSassumed} \label{pg:strictCSassumed} 
Since $\bar{X}$ and $\bar{Z}$ are mutually orthogonally diagonalizable,
we obtain
\[
\bar X - \bar Z = W+ \cA^*(\bar{y}) = 
\begin{bmatrix} V & \bar{V} \end{bmatrix}
\begin{bmatrix} R & 0 \\ 0 & -S \end{bmatrix}
\begin{bmatrix} V & \bar{V} \end{bmatrix}^T, \quad R\succ 0,\,S\succ 0,
\]
where $\bar{X} = VRV^T$ and $\bar{Z} = \bar{V}S\bar{V}^T$.

Recall the steps for computing the Jacobian in \Cref{sect:ADDF} and the
rotation operator in~\cref{eq:Rrotat}. Since we are now assuming that the matrix $\bar{X}-\bar{Z} = W+ \cA^*(\bar{y})$ is nonsingular, then the matrices of eigenvectors in~\cref{eq:Rrotat} are given by $V_{\alpha} = V$, $V_\beta = 0$ and $V_{\gamma} = \bar{V}$. Hence, the rotation operator we are interested in now is  $\cR_{\bar{X}-\bar{Z}}(\cdot)=\begin{bmatrix} V & \bar{V} \end{bmatrix}
(\cdot)\begin{bmatrix} V & \bar{V} \end{bmatrix}
^T$.
We now closely observe how the~$(i,j)$-th element of the Jacobian in \cref{eq:eachJacCol} is evaluated.
Let $T_j$ be the matrix defined in~\eqref{def:Ti}. Then 
\begin{equation}
\label{eq:expandingIJ}
\begin{array}{rcl}
&& \trace(A_i \mathcal{R}_{\bar{X}-\bar{Z}}(T_j)) \\
& = &
\left\<   A_i , \begin{bmatrix} V &\bar{V} \end{bmatrix}  
T_j \begin{bmatrix} V &\bar{V} \end{bmatrix}^T
\right\>  \\
& = &
\left\<  \begin{bmatrix} V &\bar{V} \end{bmatrix}^T A_i 
\begin{bmatrix} V &\bar{V} \end{bmatrix} ,   
T_j
\right\>  \\
& = & 
\left\<  
\begin{bmatrix} V^T A_i V &V^T A_i \bar{V} \\
\bar{V}^T A_i V & \bar{V}^T A_i \bar{V} \end{bmatrix} ,   
 \begin{bmatrix}
V^TA_j V & \Omega_{\alpha\gamma} \circ V^TA_j \bar{V} \\
\left( \Omega_{\alpha\gamma} \circ V^TA_j \bar{V} \right)^T & 0
\end{bmatrix}
\right\>
\\
& = & 
\left\<  
\begin{bmatrix} V^T A_i V &V^T A_i \bar{V} \\
\bar{V}^T A_i V & 0 \end{bmatrix} ,   
 \begin{bmatrix}
V^TA_j V &\Omega_{\alpha\gamma} \circ V^TA_j \bar{V} \\
\left( \Omega_{\alpha\gamma} \circ V^TA_j \bar{V} \right)^T & 0
\end{bmatrix}
\right\> .
\end{array}
\end{equation}
\linelabel{line:identargs} \label{pg:identargs}
Note that the two arguments in the last trace inner product in \cref{eq:expandingIJ} share the analogous block structure, with the difference arising from both the use of $A_i$ versus $A_j$ and an element-wise (Hadamard) scaling applied to the off-diagonal block.
\Cref{lemma:JacodianReducedtoVec} below links the degeneracy of the optimal point $\bar{X}$ to the invertibility of the Jacobian of $F$ at $\bar{y}$.

\begin{lemma} 
\label{lemma:JacodianReducedtoVec}
Let~$D\in \Snpp$ be a diagonal matrix, and 
let~$\{x_1,\ldots,x_m\} \subseteq \Rn$ be given. 
Let~${U} =  \begin{bmatrix} x_1 & x_2 & \cdots & x_m\end{bmatrix}$.
Then 
$\rank({U}) = \rank({U}^T {U}) = \rank ({U}^T D {U})$.
\end{lemma}

Now we use \eqref{eq:expandingIJ}, \Cref{lemma:JacodianReducedtoVec},
and  \Cref{lemma:NondegenChar} to characterize the singularity of the
Jacobian of $F$ evaluated at an optimal solution.
We note that \cite[Theorem 3.1]{AlHaOv:95} addresses 
nonsingularity of Jacobians of linear \SDPp s
under nondegeneracy assumptions.
\linelabel{line:thmsingjacob} \label{pg:thmsingjacob}
\begin{theorem}
\label{thm:BadGoodJac}
Let $(\bar{X}, \bar{y}, \bar{Z})$ satisfy the optimality condition \eqref{eq:KKTprobone} and the strict complementarity condition.
Then $\bar{X}$ is degenerate if, and only if, the Jacobian of $F$ at $\bar{y}$ is singular.
\end{theorem}
\begin{proof}
Let $\bar{X}$ be the optimal point of \cref{eq:DNNparam}.
Let \[
D = \Diag \left( \svec \left(  
\begin{bmatrix}
M & \frac{1}{\sqrt{2}} \Omega_{\alpha\gamma}(\bar{X}) \\ \frac{1}{\sqrt{2}} \Omega_{\alpha\gamma}(\bar{X})^T & M
\end{bmatrix} \right) \right) \in  \Sc^{t(n)}_{++} ,
\] 
where $M = \frac{1}{\sqrt{2}} E + \left(1-\frac{1}{\sqrt{2}}\right)I$ and  where we include the dependence on $\bar{X}$ of $\Omega$ given as in~\Cref{eq:defOmega}.
Let $X_i := \begin{bmatrix} V^T A_i V &V^T A_i \bar{V} \\
\bar{V}^T A_i V & 0 \end{bmatrix}$, 
and let $x_i:= \svec \left( X_i \right)$.
We recall the definition of $T_j$ in \cref{def:Ti} and note that 
\[ 
\svec \left( T_j \right) = D x_j .
\]
We then observe the last inner product in \cref{eq:expandingIJ}:
\[
\left\< X_i, T_j  \right\>
= \left\< \svec(X_i) , \svec(T_j) \right\> = \<x_i, Dx_j\>.
\]
Now we form $U :=  \begin{bmatrix} x_1 & x_2 & \cdots & x_m\end{bmatrix} \in \R^{t(n) \times m}$. 
Then, $\forall i,j$, we have 
\[
(U^TDU)_{i,j} 
=
\left(
\begin{bmatrix} x_1^T \\ \vdots \\ x_m^T \end{bmatrix}
\begin{bmatrix} Dx_1 \cdots Dx_m   \end{bmatrix}
\right)_{i,j} = x_i^TDx_j = \trace(A_i \mathcal{R}_{\bar{X}}(T_j)).
\]
Therefore we conclude 
\[
\begin{array}{rcll}
\bar{X} \text{ is degenerate } 
& \iff & \rank(U) < m,  & \text{by \cref{eq:RorateAis},}\\
&\iff &U^TDU \text{ is singular, } & \text{by
\Cref{lemma:JacodianReducedtoVec},} \\
&\iff & \text{Jacobian of $F$ at $\bar{X}$ is singular.} & 
\end{array}
\]
\samepage
\end{proof}

Recall the sufficient conditions for producing a nondegenerate solution
given in \Cref{prop:relintNodegen,prop:twonondeg}. Therefore,
any projection point $\bar{X}$ that satisfies the conditions in
\Cref{prop:relintNodegen,prop:twonondeg} yields a nonsingular Jacobian.

\subsection{Nondegeneracy of the Elliptope and Degeneracy of the Vontope}
\label{sec:degenofMCQAP}

We study degeneracies of two classes of spectrahedra; the elliptope (the set of correlation matrices), and the vontope (feasible region of the \SDP relaxation of the quadratic assignment problem, \QAPp).
We introduce these sets to illustrate how degeneracy interacts with the performance of \Cref{alg:1} in \Cref{sec:NumericsMCQAP}.
We exhibit the result from \cite{PatakiSVW:99} that the elliptope has only nondegenerate points; 
however all vertices of the vontope are degenerate before \FRp,
and some vertices  of the vontope are degenerate even after \FRp.

\index{vontope}

\begin{example} [{Elliptope, \cite[Thm 3.4.2]{PatakiSVW:99}}]
\label{eq:ellipt}
We consider the problem of finding the nearest correlation matrix:
\[
\min \left\{ \frac 12 \|X - W \|^2_F \,:\, \diag(X)=e, \, X\succeq
0\right\},
\]
where $e$ is the vector of all ones.
The feasible region of the above problem is called the 
\textdef{elliptope}. Note that the elliptope is the feasible
region of the \SDP relaxation of the max-cut problem.
Every point in the elliptope is nondegenerate. 
\end{example}

\begin{example}[Vontope, \cite{KaReWoZh:94}]
Let \textdef{$\Pi_n$} be the set of $n$-by-$n$ permutation matrices. For $X\in \Pi_n$, let
\[
Y_X = y_Xy_X^T, \text{ where }
y_X = \begin{pmatrix}
1\cr \kvec X
    \end{pmatrix} \in \mathbb{R}^{n^2+1},
\]
be the lifted matrix.  Here we index the rows and columns of a matrix starting from $0$.
The lifting process gives rise to the following feasible region for the \SDP relaxation:
\begin{equation}
\label{eq:QAPset}
\textdef{$\cF_{\QAP}$} := 
\left\{Y \in \Sc^{n^2+1}_+ : 
\begin{array}{l}
G_J(Y) = e_0,
\bdiag(Y)=I_n,
\odiag(Y)=I_n, \\
Y_{0,j} = Y_{j,j}, \forall j=1,\ldots,n^2 +1 
\end{array}
\right\}.
\end{equation}
Here, \textdef{$G_J$} : $\Sc^{n^2+1} \to \R^{|J|}$ is a linear map that 
chooses the elements in the index set $J$ that correspond to the $(0,0)$-element of $Y$, the off-diagonal elements of the $n$-by-$n$
diagonal blocks, and the diagonal elements of the $n$-by-$n$
off-diagonal blocks; $\bdiag : \Sc^{n^2+1} \to \Sn$ is the linear map that sums the $n$-by-$n$ diagonal blocks;
and  $\odiag : \Sc^{n^2+1} \to \Sn$  is the linear  map  defined by
$\odiag(Y) = \left( \trace ( \hat{Y}^{i,j} ) \right)_{i,j} $, where $\hat{Y}^{i,j}$ is the $(i,j)$-th $n$-by-$n$ block submatrix in $Y$; 
see \cite{KaReWoZh:94} for details on the construction of $G_J,\bdiag$ and $\odiag$.

We remark that the expression in \eqref{eq:QAPset}  contains redundant linear constraints. It is well-known that the \SDP relaxation of the \QAP fails strict
feasibility \cite{KaReWoZh:94} and so we employ \FR and work in a smaller space. 
Let 
\[
H = \begin{bmatrix} e^T\otimes I_n \cr I_n\otimes e^T\end{bmatrix} 
\in \R^{2n\times n^2}, \ 
K = \begin{bmatrix} -e & H \end{bmatrix} \in \R^{2n\times (n^2+1)} ,
\]
and let $\hat V \in \R^{(n^2+1)\times ((n-1)^2+1)}$ 
be the matrix with orthonormal columns that spans 
$\nul(K)$.\footnote{Note that the last row of $K$ is linearly 
dependent and, for efficiency and accuracy,
is best ignored when finding the nullspace.
}
\FR leads to the following constraints:
\begin{equation}
\label{eq:QAPsetFR}
\textdef{$\cF^{\FR}_{\QAP}$} := 
\{R \in \Sc^{(n-1)^2+1} : G_{\hat{J}} (\hat V R\hat V^T) = e_0, \, R\succeq 0 \},
\end{equation}
where $G_{\hat{J}}$ : $\Sc^{n^2+1} \to \R^{|\hat{J}|}$  a newly defined surjective linear map that chooses indices in $\hat{J}$ such that $\hat{J}\subsetneq J$.
This aligns with the fact that \FR reveals implicit redundant constraints. 
It is known that the number of equality constraints reduces to $n^3-2n^2+1$ after \FRp ; see~\cite{KaReWoZh:94}.\footnote{
The last column of off-diagonal blocks and the $(n-2,n-1)$ off-diagonal
block are linearly dependent and are ignored, see~\cite{GHILW:20,KaReWoZh:94}.
} 

We now discuss the degeneracy of each lifted matrix $Y_X = y_Xy_X^T = \hat V R_X\hat V^T$, where $X\in \Pi_n$. Due to the orthonormality of $\hat{V}$, we get
\[
R_X = \hat V^T Y_X \hat V \in \Sc^{(n-1)^2+1}.
\]
We note that $\rank(R_X)=1$.
We let $\{A_i\}_{i=1}^{n^3-2n^2+1}\subseteq \Sc^{(n-1)^2+1}$ be the
set of matrices defining data matrices for the affine constraints.
Hence the linear dependence of the matrices of the set
\cref{eq:RorateAis} can be analyzed by considering their first columns.  
For $n\geq 3$, we observe that the vectors
\[
\left\{ 
\begin{pmatrix} V_X^TA_iV_X \\ \bar{V}_X^TA_i V_X \end{pmatrix} 
\right\}_{i=1}^{n^3-2n^2+1}
\subseteq \R^{(n-1)^2+1}, \quad 
{n^3-2n^2+1} > (n-1)^2+1, \ 
n\geq 3,
\]
are linearly dependent. 
This proves that the rank-one vertices that arise from $\Pi_n$ remains  degenerate after \FRp .
\end{example}

\begin{remark}
\label{rem:Asuppnotat}
If we replace $\Snp$ with $\Rnp$, the set $\cF$ reduces to a polyhedron and the discussion on the degeneracy simplifies.
The degeneracy status of a point $x$ in a polyhedron can be confirmed by evaluating the rank of 
$A(:,\supp(x))$, where $\supp(x)$ $=\{i: x_i \ne 0\}$ denotes the support of $x$; see~\cite{PatakiSVW:99}. 
The performance of the proposed algorithm in \cite{CensorMoursiWeamsWolk:22} is also affected by the degeneracy of the optimal point. 
Moreover every point of $\cF$ as a polyhedron is degenerate in the absence of strict feasibility.
\end{remark}

\section{Numerical Experiments}
\label{sec:Numerics}

To illustrate the effects on convergence and degeneracy, we now present
multiple experiments using diverse spectrahedra $\cF$ with various
ranges of values for the \emph{singularity degree, $\sd(\cF)$}, and
for the \emph{implicit problem singularity, $\iips(\cF)$}. In our algorithm,
dual feasibility and complementary slackness are satisfied exactly.
Therefore, we use the following $\beps^k\in \R_+$ 
to denote the relative residual of
the optimality conditions at iteration $k$:
\index{implicit problem singularity, $\iips$}
\index{$\iips$, implicit problem singularity} 
\index{$\beps^k$, relative residual vector}
\index{relative residual vector, $\beps^k$}
\[
	\beps^k := \min \left\{1,\frac{\|F(y^k)\|}{1 + \|b\|}\right\}=:
 \alpha_k 10^{-t_k},\, 1\leq \alpha_k < 10.
\]
We denote the condition number of the Jacobian of $F$ 
at $y^k$ as $\cond(J_k)$, and let
\[
\cond(J_k) = \beta_k 10^{s_k},\, 1\leq \beta_k < 10.
\]
We stop Algorithm~\cref{alg:1} once 
\[
(i)~~\beps^k \leq 10^{-13} \text{  or   } (ii)~~s_k + t_k > 16 
\text{  or } (iii)~~k > 2000.
\]

If condition (i) holds, then the we consider the \BAP problem is solved. 
If condition (ii) holds, then we consider the optimal solution
of the \BAP problem as being degenerate. In our algorithm, if (ii) or
(iii) hold, then we conclude that a small eigenvalue for the Jacobian
exists and we assume that strict feasibility fails.\footnote{Note that
by \Cref{rem:characdeggeneric}, nondegeneracy holds for
our problem generically.} And,
by looking at the nonzero elements of
an eigenvector associated to the \emph{smallest eigenvalue} 
we get information on an exposing vector; and
we identify constraints that give rise to the failure of strict
feasibility. This solves an auxiliary system for a \FR step, 
see~\Cref{prop:thmalt}.
Using the information on the exposing vector,
we then solve a \textdef{reduced auxiliary system}, using a
\textdef{Gauss--Newton} approach\footnote{\url{https://github.com/j5im/FacialReductionSpectrahedron}}. This results in a \FR step.
Following this, we remove the
redundant constraints that arise from the \FR step. We repeat
until strict feasibility holds. 

Numerical experiments are conducted with \matlab R2023b on a Windows 11
PC with Intel(R) Core(TM) i5-10210U CPU @ 1.60GHz, RAM 16.0GB.

\subsection{Comparison With(out) Strict Feasibility}

As expected, our tests in \Cref{table:SnSold}, show that
\Cref{alg:1}  performs exceptionally well for instances with strict
feasibility but struggles when strict feasibility fails.
In fact, we observe that \Cref{alg:1} achieves
the relative precision of $10^{-7}$ in under $7$ iterations
when strict feasibility holds. In contrast, when strict
feasibility fails and \Cref{alg:1} converges, hundreds 
of iterations are needed to reach the desired precision. In  \Cref{table:SnS}, we repeat the same experiment setting a relative precision tolerance of $10^{-13}$ and allowing $2000$  iteration limit. In this case, \Cref{alg:1} never reached the desired relative precision within the maximum number of iterations.
\begin{table}[htp!]
\begin{center}
\begin{tabular}{|l|c|c|c|c|c|} 
 \hline 
n& 10& 20& 50& 100\\ 
 \hline 
Slater& 100\%& 100\%& 100\%& 100\%\\ 
  
No Slater& 55\%& 50\%& 50\%& 25\%\\ 
 \hline 
\end{tabular}

\end{center}
\caption{$20$ randomly generated problems~\Cref{eq:DNNparam};  
\% converged $\beps^k \leq 10^{-8}, k\leq 1000$.}
\label{table:SnSold}
\end{table}

\begin{table}[htp!]
\begin{center}
\begin{tabular}{|l|c|c|c|c|c|} 
 \hline 
n& 10& 20& 50& 100\\ 
 \hline 
Slater& 100\%& 100\%& 100\%& 100\%\\ 
  
No Slater&  0\%&  0\%&  0\%&  0\%\\ 
 \hline 
\end{tabular}

\end{center}
\caption{$20$ randomly generated problems~\Cref{eq:DNNparam};  
\% converged $\beps^k \leq 10^{-13}, k\leq 2000$.}
\label{table:SnS}
\end{table}

We now look at the case where the singularity degree $\sd(\cF)=1$, while 
the implicit singularity $\iips(\cF)$ varies.

\subsubsection{$\iips(\cF) = 1$}
\label{sect:iips1}

We use a spectrahedra with singularity degree $1$ and $n=15, m=7$. 
The singularity degree is obtained by
constructing an exposing vector as a linear combination of $5$
out of the $7$ constraints of the problem. \Cref{alg:1} is used
to monitor the eigenvalues of the Jacobian of $F$ at every iteration $k$,
see~\Cref{fig:eigsexp1}. We observe that only one of the eigenvalues tends to~$0$. After $452$ iterations the method reaches a relative residual of
$9.9567 \times 10^{-8}$, while the condition number of the Jacobian is
$7.0236\times 10^{12}$. Therefore the algorithm stops and indicates
which of the
constraints are causing strict feasibility to fail. 
After applying \FR and removing the single (implicit) 
redundant constraint found,
the algorithm now succeeds and converges to a point with a
relative residual of $1.0231\times 10^{-15}$ in only $8$ iterations,
see~\Cref{tab:exp1one}.

\begin{figure}[ht!]
\centering
\includegraphics[height=.4\textwidth]{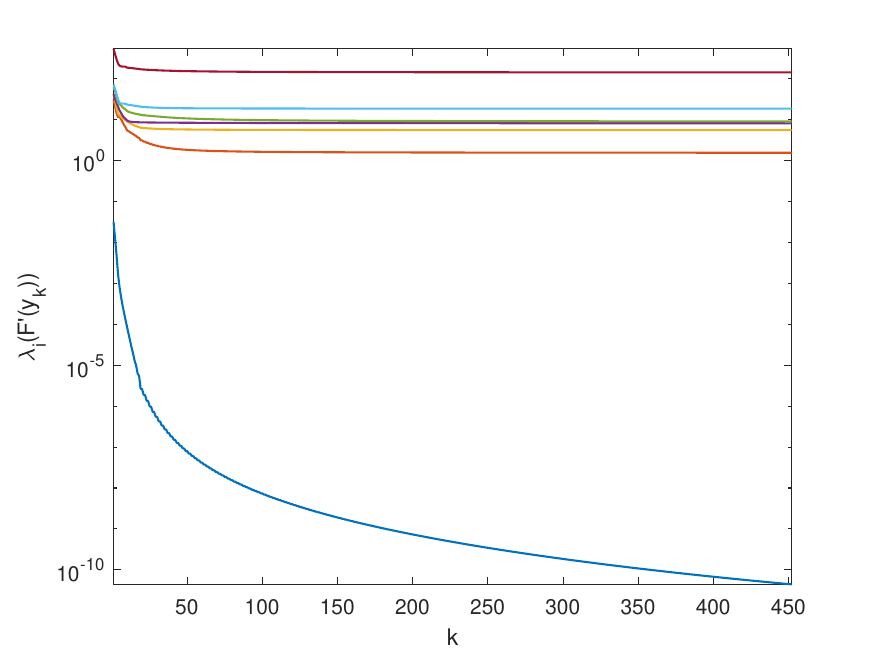}
\caption{Changes in eigenvalues of Jacobian of $F$ for
spectrahedron in \Cref{sect:iips1}.}
\label{fig:eigsexp1}
\end{figure}

\begin{table}[ht!]\centering
\setlength\extrarowheight{4pt}
\begin{tabular}{lcccccc} \toprule
&    $n$ & $m$ & $\beps^k$(rel. res.)  &$\mathrm{cond}(F'(y_k))$ & $\lambda_n(y_k)$ & $k$ \\[5pt]
  \hline 
Before \textbf{FR} &15 & 7 & 9.9567e-08 & 7.0236e+12 & -1.7238e-16 & 452\\[5pt] 

  \hline 
After \textbf{FR} &15 & 6 & 1.0231e-15 & 198.08 & 2.5515e-17 & 8\\[5pt] 

  \hline 
\end{tabular}

\caption{Spectrahedron in \Cref{sect:iips1}; at final iteration $k$;
		before and  after \FR iters}
\label{tab:exp1one}
\end{table}

\subsubsection{$\iips(\cF) > 1$}
\label{sect:iipsgtone}

In our second experiment, see~\Cref{tab:exp1two},
we work with data obtained from a \SDP relaxation of the
protein side-chain positioning problems, e.g.,~\cite{BurkImWolk:20}. 
The spectrahedron we are considering has singularity
degree $1$, but the implicit problem singularity is  greater than $1$,
i.e.,  there are more than $1$ redundant constraints after applying \FR.
In particular, the dimension of the space is $n=35$ and the number of
constraints is $m=75$.  By running our algorithm, we observe that a
large number of eigenvalues of the Jacobian tend to $0$ along the
iterations (see~\Cref{fig:exp2}). After applying \FR, we reduced the
dimension of the problem to $n=10$ and the number of constraints to
$m=22$. In the next run of the algorithm, only one eigenvalue of the
Jacobian tends to $0$, but we detect that a second iteration of \FR is
needed. This time, we reduce $n$ to $9$ and we remove $6$ more redundant
constraints, resulting in $m=16$. The next time we apply our algorithm,
the method converges to the solution in $18$ iterations, 
see~\Cref{tab:exp1two}.

\begin{figure}[ht!]\centering
	 \includegraphics[height=.35\textwidth]{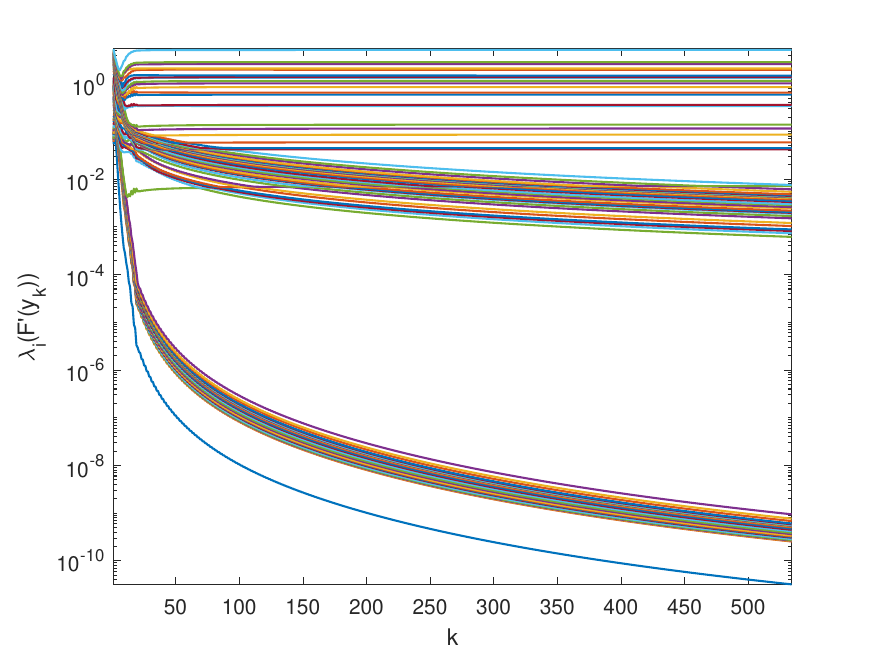}\hfill
	 \includegraphics[height=.35\textwidth]{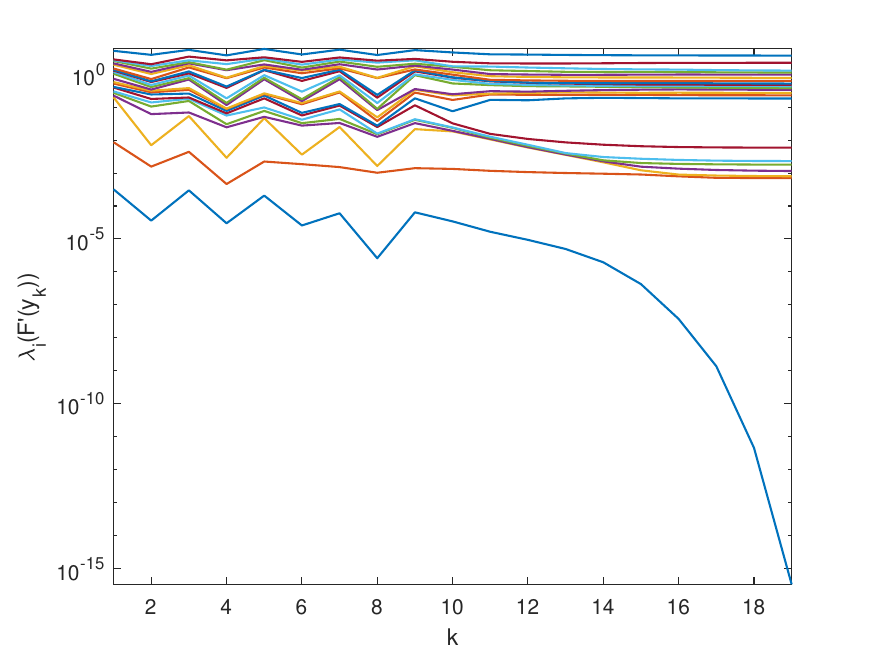}
\caption{Iterations $k$ vs eigenvalues; spectrahedron in
\Cref{sect:iipsgtone}; 
		before and  after one \FR iteration}
\label{fig:exp2}
\end{figure}

\begin{table}[ht!]\centering
\setlength\extrarowheight{4pt}
\begin{tabular}{lcccccc} \toprule
&    $n$ & $m$ & $\beps^k$(rel. res.)  &$\mathrm{cond}(F'(y_k))$ & $\lambda_n(y_k)$ & $k$ \\[5pt]
  \hline 
Before \textbf{FR} &35 & 76 & 8.5351e-08 & 1.0060e+12 & -1.6941e-15 & 534\\[5pt] 

  \hline 
After \textbf{FR}  1 &10 & 22 & 7.6363e-04 & 1.8739e+16 & -5.2097e-16 & 19\\[5pt] 

  \hline 
After \textbf{FR}  2 &9 & 16 & 8.7202e-14 & 16103.37 & -5.6900e-16 & 18\\[5pt] 

  \hline 
\end{tabular}

\caption{Spectrahedron in \Cref{sect:iipsgtone}; at final iteration $k$;
		before and  after \FR iters}
\label{tab:exp1two}
\end{table}

\subsection{Experiments with Elliptope and Vontope}
\label{sec:NumericsMCQAP}

In this section we address the importance of strict feasibility 
and degeneracy on the performance of
\Cref{alg:1}. We consider the elliptope and vontope cases.
Furthermore we compare the performance of \Cref{alg:1}
with the interior point solver 
SDPT3\footnote{\url{https://www.math.cmu.edu/~reha/sdpt3.html},
version SDPT3 4.0~\cite{ToddTohTut:96b}.}. 

\index{vontope}

From \Cref{sec:degenofMCQAP} we recall
that the \MC problem satisfies strict feasibility and every point
of the elliptope, the feasible set, is nondegenerate; see~\Cref{eq:ellipt}.
The results from the \MC problem are displayed in the line
labelled `Elliptope' in~\Cref{table:MCQAP}.
As for the \QAPp, without \FRp, 
the \SDP relaxation of \QAP fails strict feasibility and all the
feasible points are degenerate.
Hence in our tests, we consider two models of the same set of instances: 
\textdef{$\cF_{\QAP}$} obtained directly by the lifting of the variables
(see \eqref{eq:QAPset}); and \textdef{$\cF^{\FR}_{\QAP}$} obtained after
\FR is applied to $\cF_{\QAP}$ (see \eqref{eq:QAPsetFR}). 
In \Cref{table:MCQAP}, \QAP ($\QAPp_\FRp$, resp.) indicates the results
obtained from $\cF_{\QAP}$ ($\cF^{\FRp}_{\QAPp}$, resp.). 

We used two settings for the choice of $W$ in the objective function. 
The first setting for $W$ forces the optimal solution $\bar{X}$ to be rank $1$. 
Recall that rank-one optimal solutions for $\QAP$ are degenerate and thus 
lead to ill-conditioned Jacobians as can be seen by the huge condition
numbers demonstrated later in~\Cref{table:MCQAP}.
The second setting chooses a random $W$. 
Followed by the discussions in \Cref{rem:characdeggeneric} and \Cref{prop:twonondeg}, $\bar{X}$ is generically nondegenerate 
when strict feasibility holds.

For SDPT3 we provided the following second-order cone formulation of
\eqref{eq:DNNparam}:
\[
\min_{X,y,t} \left\{ \  t \ : \svec(X) + y = \svec(W), \ ||y||_2\le t , \  X\in \cF \ \right\}.
\]
The default settings for SDPT3 were used for the tests.

Each line of \Cref{table:MCQAP} reports on the average of $10$ instances,
problem order $n=10$. 
The meaning of the header names used in \Cref{table:MCQAP} is as follows:
\begin{enumerate}[(i)]
\itemsep0em 
\item 
The headers
`pf', `df' and `cs' under Semi-Smooth Newton refer to the average of the
primal feasibility, dual feasibility and complementarity, respectively,
introduced in~\eqref{eq:KKTproboneFR}. The df includes both the linear
dual feasibility and the violation of semidefiniteness. Both are
essentially zero up to roundoff error of the arithmetic.
Note that the values $10^{-15}$ and smaller for pf and df are essentially zero (machine precision).
The headers pf, df and cs under SPDT3 refer to
the solver outputs, pinfeas, dinfeas and gap, respectively.
\item $k$ is the average number of iterations.
\item time is the average run time in cpu-seconds.
\item $\mathrm{cond}(F^\prime(y^k))$ is the average condition number of
the Jacobian $(F^\prime(y^k))$; we only have this metric for the
semismooth Newton method.
\end{enumerate}

\begin{table}[h!]
\centering
	\scalebox{0.7}{
		\begin{tabular}{c|l|cccccc|ccccc}
\toprule
\multirow{2}{*}{$W$ Generation} & \multirow{2}{*}{Problem} & \multicolumn{6}{c|}{Semi-Smooth Newton} & \multicolumn{5}{c}{SDPT3}  \\
 & & pf & df & cs &   $k$ & time  & $\mathrm{cond}(F^\prime(y^k))$ & pf & df & cs &    $k$ & time\\
\midrule
& Elliptope & 9e-13 & 9e-16 &  2e-16 &  6.8 & 4e-02  & 3e+00 & 4e-12 & 6e-12 &  2e-07 & 15.5 & 2e-01 \\$W$, $\rank(\bar{X})=1$ & $\QAPp_{\FR}$ & 4e-07 & 2e-15 &  1e-16 &  7.5 & 7e+00  & 4e+15 & 5e-10 & 1e-09 &  9e-06 & 17.9 & 7e+01 \\& \QAP & 8e-09 & 3e-15 &  1e-16 &  8.6 & 2e+01  & 4e+14 & 5e-10 & 5e-09 &  1e-05 & 18.9 & 6e+01 \\\midrule
& Elliptope & 3e-12 & 1e-15 &  6e-17 &  6.3 & 1e-02  & 2e+00 & 1e-11 & 6e-12 &  3e-08 & 11.5 & 9e-02 \\random $W$ & $\QAPp_{\FR}$ & 2e-12 & 3e-15 &  7e-17 & 20.6 & 2e+01  & 3e+05 & 5e-10 & 5e-10 &  7e-07 & 13.9 & 5e+01 \\& \QAP & 1e-07 & 5e-13 &  3e-16 & 537.9 & 1e+03  & 6e+11 & 1e-08 & 2e-09 &  1e-06 & 17.3 & 7e+01 \\\bottomrule
\end{tabular}

}
\caption{\Cref{alg:1} and SDPT3 on: Elliptope and Vontope; $n=10$;
}
\label{table:MCQAP}
\end{table}

We now discuss the results in \Cref{table:MCQAP}.
We start with the Semi-Smooth Newton, \Cref{alg:1}.
The `pf' column clearly indicates that
the degeneracy of the optimal point $\bar{X}$ plays a significant role.
Except for random $W$ with $\QAPp_\FRp$, the `pf' values for the \QAP
 are poor. This correlates with the condition number values;
see also the discussion in \Cref{sec:FailureReg}.
The condition numbers of the Jacobian near optimal points,
$\cond(F^\prime(y^k))$, become ill-conditioned when strict feasibility
fails and the optimal solution is degenerate. 
The good measures for  `df' and `cs' of Semi-Smooth Newton method is attributed to the construction of \Cref{alg:1}.
We observe that the average number of iterations for \QAP is high when a random $W$ is constructed, 
since the algorithm's difficulty in achieving the stopping tolerance~$\epsilon$, set to $10^{-7}$.

\linelabel{line:sdpt3diff} \label{pg:sdpt3diff}
SDPT3 exhibits overall strong performance on all instances.
However, the `df' and `cs' values under SDPT3 are weaker compared to the Semi-Smooth Newton. This is due to interior point methods aiming to satisfy the first-optimality condition simultaneously throughout the process. 
We also observe that the number of iterations is higher when the
optimal solutions are set to be degenerate.

\Cref{alg:1} has a superior performance for \MC problems as all
components of the optimality conditions are satisfied with near machine
accuracy. 
This confirms that the status of the optimal solution plays an
important role when it comes to the performance of \Cref{alg:1}.
In addition, preprocessing the instances so that they satisfy strict
feasibility is important as seen by problems failing strict feasibility only
contain degenerate points; see~\Cref{thm:AiVdep}.


\section{Conclusions}
\label{sect:concl}

We presented and analyzed a semismooth Newton method for the best
approximation problem, the projection problem, for spectrahedra. 
We showed that nondegeneracy is needed for the semismooth Newton method 
to perform well.  
We used the unbounded dual optimal set in the absence of 
a regularity condition to explain the lack of good performance.
Moreover, we showed that the absence of strict feasibility results in
degeneracy and ill-conditioning of the Jacobian at optimality.
Our empirics illustrate the importance of strict feasibility.
In particular, we studied the degeneracy for the elliptope and vontope.

Though we concentrated on \SDPp, many current relaxations for hard
problems involve the doubly nonnegative, \DNNp, cone, i.e.,~$\DNN =
\Snp\cap \R^{n\times n}_+$. In particular, splitting methods efficiently
exploit this intersection of two cones and facial reduction often
provides a natural efficient splitting,
e.g.,~\cite{OliveiraWolkXu:15,WassersteinAlfakihetalRevised:24,GHILW:20}.
It seems that the
results we obtained from the Newton method for the \BAP would extend to
applying splitting methods to
feasible sets of the type $\cL\cap \DNN$. 

\subsection*{Data Availability Statement}
The results and data used in this paper is generated using our \matlab codes.
These are publicly available at the link:\\
\url{https:\\www.math.uwaterloo.ca/~hwolkowi/henry/reports/CodesProjDegSingDegJul2024.d/
}


\cleardoublepage
\addcontentsline{toc}{section}{Index}
\begin{theindex}
	
	\item $(x,y)$, open interval, \hyperpage{20}
	\item $C^+$, nonnegative polar cone of $C$, \hyperpage{6}
	\item $C^{\perp}$, orthogonal complement of $C$, \hyperpage{6}
	\item $F( y): = \cA P_{\Snp}(W+\cA^*  y) - b$, \hyperpage{10}
	\item $F_f( \hat  y): = {\mathcal  A} P_f( W+ {\mathcal  A} ^* \hat  y) - b$, 
	\hyperpage{13}
	\item $G_J$, \hyperpage{25}
	\item $K\unlhd C$, face, \hyperpage{6}
	\item $P_S$, projection onto a nonempty closed convex set $S$, 
	\hyperpage{5}
	\item $W\in \Sn$, data, \hyperpage{3}
	\item $X\succeq _{\mathcal  X} 0$, \hyperpage{4}
	\item $X^* = {\mathcal  V} (R^*(\widebar  W))$, \hyperpage{12}
	\item $X^*$, optimum, \hyperpage{3}
	\item $\Delta(X)$, Moreau regularization of $\iota_{\Snn}$, 
	\hyperpage{5}
	\item $\Pi \svec  W_i$, \hyperpage{8}
	\item $\Pi _n$, \hyperpage{25}
	\item $\beps^k$, relative residual vector, \hyperpage{27}
	\item $\cF = \cL\cap K$, feasible set, \hyperpage{4}
	\item $\cOn$, orthogonal matrices, \hyperpage{5}
	\item $\cdot^\dagger$, Moore-Penrose generalized inverse, 
	\hyperpage{12}
	\item $\face  (X,\mathbb  {S}_+^n)$, \hyperpage{6}
	\item $\iips$, implicit problem singularity, \hyperpage{7}, 
	\hyperpage{27}
	\item $\iota_S$, indicator function, \hyperpage{5}
	\item $\mathcal  {R}_{Y}$, \hyperpage{19}
	\item $\maxsd(\cF)$, max-singularity degree of $\cF$, \hyperpage{7}
	\item $\null A$, null space of $A$, \hyperpage{20}
	\item $\phi(y,Z)$, dual functional, \hyperpage{10}
	\item $\relint$, relative interior , \hyperpage{6}
	\item $\sd(\cF)$, singularity degree of $\cF$, \hyperpage{3}, 
	\hyperpage{7}
	\item $\svec  : \mathbb  {S}^k\to \mathbb  {R}^{t(k)}$, \hyperpage{8}
	\item $\widebar  F( y): = \bar  {\mathcal  A} P_{\mathbb  {S}_+^r}(\widebar  W+\bar  {\mathcal  A} ^* y) - \bar  b$, 
	\hyperpage{12}
	\item $\widebar  W = V^TWV\in \mathbb  {S}^r$, \hyperpage{12}
	\item $f := \face  ({\mathcal  F} ,\mathbb  {S}_+^n)$, the minimal face of $\mathbb  {S}_+^n$ containing~${\mathcal  F} $, 
	\hyperpage{12}
	\item $f^{\Delta }$, conjugate face of $f$, \hyperpage{6}
	\item $p^* = d^*$, zero duality gap, \hyperpage{20}
	\item $p^*$, optimal value, \hyperpage{3}
	\item $t(n) = n(n+1)/2$, triangular number, \hyperpage{8}
	\item ${\mathcal  F} ^{\textbf  {FR}\,}_{\textbf  {QAP}\,}$, 
	\hyperpage{26}, \hyperpage{30}
	\item ${\mathcal  F} _{\textbf  {QAP}\,}$, \hyperpage{25}, 
	\hyperpage{30}
	\item ${\mathcal  S} = \{ y\in \mathbb  {R}^m: F(y) = 0\}$, 
	\hyperpage{21}
	\item ${\mathcal  S} _\lambda $, \hyperpage{15}
	\item ${\mathcal  V} (R) := VRV^T$, \hyperpage{12}
	\item ${\rm  vec}$, \hyperpage{8}
	\item \BAPp, best approximation problem, \hyperpage{3}
	\item \FRp, facial reduction, \hyperpage{7}
	\item \KKTp, Karush-Kuhn-Tucker, \hyperpage{13}
	\item \QAPp, quadratic assignment problem, \hyperpage{20}
	
	\indexspace
	
	\item adjoint, \hyperpage{7}
	\item auxiliary system, \hyperpage{7}, \hyperpage{9}
	
	\indexspace
	
	\item best approximation problem, \mbox  {\bf  BAP}, \hyperpage{3}
	
	\indexspace
	
	\item conjugate face of $K$, $K^{\Delta}$, \hyperpage{6}
	\item correlation matrix, \hyperpage{4}
	
	\indexspace
	
	\item data, $W\in \mathbb  {S}^n$, \hyperpage{3}
	\item degenerate point, \hyperpage{8}
	\item dual functional, $\phi(y,Z)$, \hyperpage{10}
	
	\indexspace
	
	\item elliptope, \hyperpage{4}, \hyperpage{25}
	\item exposing vector, \hyperpage{7}
	
	\indexspace
	
	\item face, $K\unlhd C$, \hyperpage{6}
	\item facial range vector, \hyperpage{6}
	\item facial reduction, \FRp, \hyperpage{7}
	\item feasible set, $\cF = \cL\cap K$, \hyperpage{4}
	
	\indexspace
	
	\item Gauss--Newton, \hyperpage{27}
	
	\indexspace
	
	\item implicit problem singularity, $\iips$, \hyperpage{7}, 
	\hyperpage{27}
	\item indicator function, $\iota _S$, \hyperpage{5}
	
	\indexspace
	
	\item Karush-Kuhn-Tucker, \textbf  {KKT}, \hyperpage{11}
	
	\indexspace
	
	\item max-singularity degree of $\cF$, $\maxsd(\cF)$, \hyperpage{7}
	\item method of alternating projections, \hyperpage{4}
	\item minimal face of $\cF$, $f = \face(\cF,\Snp)$, \hyperpage{12}
	\item Moore-Penrose generalized inverse, $\cdot^\dagger$, 
	\hyperpage{12}
	\item Moreau regularization of $\iota_{\Snn}$, $\Delta(X)$, 
	\hyperpage{5}
	
	\indexspace
	
	\item nondegenerate, \hyperpage{8}
	
	\indexspace
	
	\item open interval, $(x,y)$, \hyperpage{20}
	\item optimal value, $p^*$, \hyperpage{3}
	\item optimum, $X^*$, \hyperpage{3}
	\item orthogonal matrices, $\cOn$, \hyperpage{5}
	
	\indexspace
	
	\item projection onto closed convex set $S$, $P_S$, \hyperpage{5}
	\item proper face, \hyperpage{6}
	
	\indexspace
	
	\item quadratic assignment problem, \QAPp, \hyperpage{20}
	
	\indexspace
	
	\item recession direction, \hyperpage{22}
	\item reduced auxiliary system, \hyperpage{27}
	\item relative interior, $\relint$, \hyperpage{6}
	\item relative residual vector, $\beps^k$, \hyperpage{27}
	
	\indexspace
	
	\item singularity degree of $\cF$, $\sd(\cF)$, \hyperpage{3}, 
	\hyperpage{7}
	\item Slater constraint qualification, \hyperpage{8}
	\item spectrahedron, \hyperpage{3, 4}
	\item spectral function, \hyperpage{5}
	
	\indexspace
	
	\item triangular number, $t(n) = n(n+1)/2$, \hyperpage{8}
	
	\indexspace
	
	\item vontope, \hyperpage{4}, \hyperpage{25}, \hyperpage{30}
	
	\indexspace
	
	\item zero duality gap, $p^* = d^*$, \hyperpage{20}
	
\end{theindex}

\def\udot#1{\ifmmode\oalign{$#1$\crcr\hidewidth.\hidewidth
	}\else\oalign{#1\crcr\hidewidth.\hidewidth}\fi} \def\cprime{$'$}
\def\cprime{$'$} \def\cprime{$'$}

\addcontentsline{toc}{section}{Bibliography}


\begin{thebibliography}{1000}
	
	\bibitem{AlHaOv:95}
	{\sc F.~Alizadeh, J.-P. Haeberly, and M.~Overton}, {\em Complementarity and
		nondegeneracy in semidefinite programming}, Math. Program., 77 (1997),
	pp.~111--128.
	
	\bibitem{MR1931309}
	{\sc A.~Auslender and M.~Teboulle}, {\em Asymptotic cones and functions in
		optimization and variational inequalities}, Springer Monographs in
	Mathematics, Springer-Verlag, New York, 2003.
	
	\bibitem{MR1797294}
	{\sc A.~Barvinok}, {\em A remark on the rank of positive semidefinite matrices
		subject to affine constraints}, Discrete Comput. Geom., 25 (2001),
	pp.~23--31.
	
	\bibitem{MR3155358}
	{\sc H.~Bauschke and V.~Koch}, {\em Projection methods: {S}wiss army knives for
		solving feasibility and best approximation problems with halfspaces}, in
	Infinite products of operators and their applications, vol.~636 of Contemp.
	Math., Amer. Math. Soc., Providence, RI, 2015, pp.~1--40.
	
	\bibitem{BiGr:74}
	{\sc A.~Ben-Israel and T.~Greville}, {\em Generalized Inverses: Theory and
		Applications}, Wiley-Interscience, 1974.
	
	\bibitem{MR1885204}
	{\sc R.~Bixby}, {\em Solving real-world linear programs: a decade and more of
		progress}, Oper. Res., 50 (2002), pp.~3--15.
	\newblock 50th anniversary issue of Operations Research.
	
	\bibitem{MR2580548}
	{\sc R.~Borsdorf and N.~J. Higham}, {\em A preconditioned {N}ewton algorithm
		for the nearest correlation matrix}, IMA J. Numer. Anal., 30 (2010),
	pp.~94--107.
	
	\bibitem{bw3}
	{\sc J.~Borwein and H.~Wolkowicz}, {\em Regularizing the abstract convex
		program}, J. Math. Anal. Appl., 83 (1981), pp.~495--530.
	
	\bibitem{BoWo:86}
	\leavevmode\vrule height 2pt depth -1.6pt width 23pt, {\em A simple constraint
		qualification in infinite-dimensional programming}, Math. Programming, 35
	(1986), pp.~83--96.
	
	\bibitem{BurkImWolk:20}
	{\sc F.~Burkowski, H.~Im, and H.~Wolkowicz}, {\em A {P}eaceman-{R}achford
		splitting method for the protein side-chain positioning problem}, INFORMS J.
	Comput., 37 (2025), pp.~962--976.
	
	\bibitem{CensorMoursiWeamsWolk:22}
	{\sc Y.~Censor, , W.~Moursi, T.~Weames, and H.~Wolkowicz}, {\em Regularized
		nonsmooth {N}ewton algorithms for best approximation with applications},
	tech. rep., University of Waterloo, Waterloo, Ontario, 2022 under revision.
	\newblock 37 pages, research report, under revision.
	
	\bibitem{MR3374759}
	{\sc Y.~Censor}, {\em Weak and strong superiorization: between
		feasibility-seeking and minimization}, An. \c{S}tiin\c{t}. Univ. ``Ovidius''
	Constan\c{t}a Ser. Mat., 23 (2015), pp.~41--54.
	
	\bibitem{MR0056264}
	{\sc A.~Charnes}, {\em Optimality and degeneracy in linear programming},
	Econometrica, 20 (1952), pp.~160--170.
	
	\bibitem{Clarke1990optimization}
	{\sc F.~Clarke}, {\em Optimization and nonsmooth analysis}, vol.~5 of Classics
	in Applied Mathematics, Society for Industrial and Applied Mathematics
	(SIAM), Philadelphia, PA, second~ed., 1990.
	
	\bibitem{MR3301316}
	{\sc M.~de~Carli~Silva and L.~Tun\c{c}el}, {\em Vertices of spectrahedra
		arising from the elliptope, the theta body, and their relatives}, SIAM J.
	Optim., 25 (2015), pp.~295--316.
	
	\bibitem{MR4614122}
	{\sc L.~Ding and M.~Udell}, {\em A strict complementarity approach to error
		bound and sensitivity of solution of conic programs}, Optim. Lett., 17
	(2023), pp.~1551--1574.
	
	\bibitem{DrusLiWolk:14}
	{\sc D.~Drusvyatskiy, G.~Li, and H.~Wolkowicz}, {\em A note on alternating
		projections for ill-posed semidefinite feasibility problems}, Math. Program.,
	162 (2017), pp.~537--548.
	
	\bibitem{DrusWolk:16}
	{\sc D.~Drusvyatskiy and H.~Wolkowicz}, {\em The many faces of degeneracy in
		conic optimization}, Foundations and
	Trends\textsuperscript{\tiny\textregistered} in Optimization, 3 (2017),
	pp.~77--170.
	
	\bibitem{MR3622250}
	{\sc M.~D\"ur, B.~Jargalsaikhan, and G.~Still}, {\em Genericity results in
		linear conic programming---a tour d'horizon}, Math. Oper. Res., 42 (2017),
	pp.~77--94.
	
	\bibitem{EckartYoung:36}
	{\sc C.~Eckart and G.~Young}, {\em The approximation of one matrix by another
		of lower rank}, Psychometrica, 1 (1936), pp.~211--218.
	
	\bibitem{facchinei2003finiteVol2}
	{\sc F.~Facchinei and J.-S. Pang}, {\em Finite-dimensional variational
		inequalities and complementarity problems}, vol.~2, Springer, 2003.
	
	\bibitem{Fr:81}
	{\sc S.~Friedland}, {\em Convex spectral functions}, Linear and Multilinear
	Algebra, 9 (1980/81), pp.~299--316.
	
	\bibitem{MR0489903}
	{\sc J.~Gauvin}, {\em A necessary and sufficient regularity condition to have
		bounded multipliers in nonconvex programming}, Math. Programming, 12 (1977),
	pp.~136--138.
	
	\bibitem{GOULART2020177}
	{\sc P.~Goulart, Y.~Nakatsukasa, and N.~Rontsis}, {\em Accuracy of approximate
		projection to the semidefinite cone}, Linear Algebra and its Applications,
	594 (2020), pp.~177--192.
	
	\bibitem{GHILW:20}
	{\sc N.~Graham, H.~Hu, H.~Im, X.~Li, and H.~Wolkowicz}, {\em A restricted dual
		{P}eaceman-{R}achford splitting method for a strengthened {DNN} relaxation
		for {QAP}}, INFORMS J. Comput., 34 (2022), pp.~2125--2143.
	
	\bibitem{HenrionMalick:11}
	{\sc D.~Henrion and J.~Malick}, {\em Projection methods for conic feasibility
		problems: applications to polynomial sum-of-squares decompositions},
	Optimization Methods and Software, 26 (2011), pp.~23--46.
	
	\bibitem{MR1918653}
	{\sc N.~J. Higham}, {\em Computing the nearest correlation matrix---a problem
		from finance}, IMA J. Numer. Anal., 22 (2002), pp.~329--343.
	
	\bibitem{MR3537883}
	{\sc N.~J. Higham and N.~Strabi\'{c}}, {\em Bounds for the distance to the
		nearest correlation matrix}, SIAM J. Matrix Anal. Appl., 37 (2016),
	pp.~1088--1102.
	
	\bibitem{MR1865628}
	{\sc J.-B. Hiriart-Urruty and C.~Lemar{\'e}chal}, {\em Fundamentals of convex
		analysis}, Grundlehren Text Editions, Springer-Verlag, Berlin, 2001.
	\newblock Abridged version of {{\i}t Convex analysis and minimization
		algorithms. I} [Springer, Berlin, 1993; MR1261420 (95m:90001)] and {{\i}t II}
	[ibid.; MR1295240 (95m:90002)].
	
	\bibitem{HaesolIm:2022}
	{\sc H.~Im}, {\em Implicit Loss of Surjectivity and Facial Reduction: Theory
		and Applications}, PhD thesis, University of Waterloo, 2023.
	
	\bibitem{im2024implicit}
	\leavevmode\vrule height 2pt depth -1.6pt width 23pt, {\em Implicit redundancy
		and degeneracy in conic program}, tech. rep., arXiv, 2403.04171, 2024.
	\newblock arXiv, 2403.04171.
	
	\bibitem{ImWolk:21}
	{\sc H.~Im and H.~Wolkowicz}, {\em A strengthened {B}arvinok-{P}ataki bound on
		{SDP} rank}, Oper. Res. Lett., 49 (2021), pp.~837--841.
	
	\bibitem{ImWolk:22}
	{\sc H.~Im and H.~Wolkowicz}, {\em Revisiting degeneracy, strict feasibility,
		stability, in linear programming}, European J. Oper. Res., 310 (2023),
	pp.~495--510.
	
	\bibitem{WassersteinAlfakihetalRevised:24}
	{\sc W.~L. Jung and H.~Wolkowicz}, {\em Exact solutions for the {NP}-hard
		{W}asserstein barycenter problem using a doubly nonnegative relaxation and a
		splitting method}, Springer Optimization and its Applications, Springer,
	Waterloo, Ontario, 2024.
	\newblock 21 pages, to appear.
	
	\bibitem{Lew:96}
	{\sc A.~Lewis}, {\em Convex analysis on the {H}ermitian matrices}, SIAM J.
	Optim., 6 (1996), pp.~164--177.
	
	\bibitem{liu2018exact}
	{\sc M.~Liu and G.~Pataki}, {\em Exact duals and short certificates of
		infeasibility and weak infeasibility in conic linear programming},
	Mathematical Programming, 167 (2018), pp.~435--480.
	
	\bibitem{MR2112861}
	{\sc J.~Malick}, {\em A dual approach to semidefinite least-squares problems},
	SIAM J. Matrix Anal. Appl., 26 (2004), pp.~272--284 (electronic).
	
	\bibitem{MR2252652}
	{\sc J.~Malick and H.~S. Sendov}, {\em Clarke generalized {J}acobian of the
		projection onto the cone of positive semidefinite matrices}, Set-Valued
	Anal., 14 (2006), pp.~273--293.
	
	\bibitem{MR4675665}
	{\sc R.~Mansour}, {\em New {A}lgorithmic {S}tructures for
		{F}easibility-{S}eeking and for {B}est {A}pproximation {P}roblems and their
		{C}onvergence {A}nalyses}, ProQuest LLC, Ann Arbor, MI, 2023.
	\newblock Thesis (Ph.D.)--University of Haifa (Israel).
	
	\bibitem{MR850380}
	{\sc N.~Megiddo}, {\em A note on degeneracy in linear programming}, Math.
	Programming, 35 (1986), pp.~365--367.
	
	\bibitem{MiSmSwWa:85}
	{\sc C.~Micchelli, P.~Smith, J.~Swetits, and J.~Ward}, {\em Constrained $l_p$
		approximation}, Journal of Constructive Approximation, 1 (1985), pp.~93--102.
	
	\bibitem{MR4832118}
	{\sc H.~Ochiai, Y.~Sekiguchi, and H.~Waki}, {\em Analytic formulas for
		alternating projection sequences for the positive semidefinite cone and an
		application to convergence analysis}, J. Math. Anal. Appl., 544 (2025),
	pp.~Paper No. 129070, 30.
	
	\bibitem{OliveiraWolkXu:15}
	{\sc D.~Oliveira, H.~Wolkowicz, and Y.~Xu}, {\em A{DMM} for the {SDP}
		relaxation of the {QAP}}, Math. Program. Comput., 10 (2018), pp.~631--658.
	
	\bibitem{PangSunSun:03}
	{\sc J.-S. Pang, D.~Sun, and J.~Sun}, {\em Semismooth homeomorphisms and strong
		stability of semidefinite and {L}orentz complementarity problems}, Math.
	Oper. Res., 28 (2003), pp.~39--63.
	
	\bibitem{ParikhBoyd:13}
	{\sc N.~Parikh and S.~Boyd}, {\em Proximal algorithms}, Foundations and
	Trends\textsuperscript{\tiny\textregistered} in Optimization, 1 (2013),
	pp.~123--231.
	
	\bibitem{GPat:95}
	{\sc G.~Pataki}, {\em On the rank of extreme matrices in semidefinite programs
		and the multiplicity of optimal eigenvalues}, Math. Oper. Res., 23 (1998),
	pp.~339--358.
	
	\bibitem{PatakiSVW:99}
	{\sc G.~Pataki}, {\em Geometry of {S}emidefinite {P}rogramming}, in Handbook OF
	Semidefinite Programming: Theory, Algorithms, and Applications, H.~Wolkowicz,
	R.~Saigal, and L.~Vandenberghe, eds., Kluwer Academic Publishers, Boston, MA,
	2000.
	
	\bibitem{RaTuWo:95}
	{\sc M.~V. Ramana, L.~Tun\c{c}el, and H.~Wolkowicz}, {\em Strong duality for
		semidefinite programming}, SIAM J. Optim., 7 (1997), pp.~641--662.
	
	\bibitem{Sremac:2019}
	{\sc S.~Sremac}, {\em Error bounds and singularity degree in semidefinite
		programming}, PhD thesis, University of Waterloo, 2019.
	
	\bibitem{SWW:17}
	{\sc S.~Sremac, H.~Woerdeman, and H.~Wolkowicz}, {\em Error bounds and
		singularity degree in semidefinite programming}, SIAM J. Optim., 31 (2021),
	pp.~812--836.
	
	\bibitem{S98lmi}
	{\sc J.~Sturm}, {\em Error bounds for linear matrix inequalities}, SIAM J.
	Optim., 10 (2000), pp.~1228--1248 (electronic).
	
	\bibitem{Su:06}
	{\sc D.~Sun}, {\em The strong second-order sufficient condition and constraint
		nondegeneracy in nonlinear semidefinite programming and their implications},
	Math. Oper. Res., 31 (2006), pp.~761--776.
	
	\bibitem{sun2002semismooth}
	{\sc D.~Sun and J.~Sun}, {\em Semismooth matrix-valued functions}, Mathematics
	of Operations Research, 27 (2002), pp.~150--169.
	
	\bibitem{ToddTohTut:96b}
	{\sc K.~Toh, M.~Todd, and R.~T{\"u}t{\"u}nc{\"u}}, {\em S{DPT}3---a {MATLAB}
		software package for semidefinite programming, version 1.3}, Optim. Methods
	Softw., 11/12 (1999), pp.~545--581.
	\newblock Interior point methods.
	
	\bibitem{ScTuWominimal:07}
	{\sc L.~Tun\c{c}el and H.~Wolkowicz}, {\em Strong duality and minimal
		representations for cone optimization}, Comput. Optim. Appl., 53 (2012),
	pp.~619--648.
	
	\bibitem{Yu2014ThePO}
	{\sc Y.-L. Yu}, {\em The proximity operator}, in Semantic Scholar, 2014.
	
	\bibitem{KaReWoZh:94}
	{\sc Q.~Zhao, S.~E. Karisch, F.~Rendl, and H.~Wolkowicz}, {\em Semidefinite
		programming relaxations for the quadratic assignment problem}, J. Comb.
	Optim., 2 (1998), pp.~71--109.
	\newblock Semidefinite programming and interior-point approaches for
	combinatorial optimization problems (Toronto, ON, 1996).
	
\end{thebibliography}
\end{document}